\documentclass[11pt, reqno]{amsart}
\usepackage{amsthm}
\usepackage{amsmath,amssymb,amsthm,mathrsfs,amscd}
\usepackage{MnSymbol}
\usepackage{color}
\usepackage[all]{xy}
\usepackage[english]{babel}
\usepackage{graphicx}
\usepackage[latin1]{inputenc}
\usepackage[T1]{fontenc}
\usepackage[final=true]{hyperref}
\usepackage{tikz}
\usetikzlibrary{patterns}
\usetikzlibrary{arrows}
\usepackage{rotating}
\usepackage{cases}
\usepackage{mathtools}
\usepackage{tikz-cd} 

\makeatletter
\newtheorem*{rep@theorem}{\rep@title}
\newcommand{\newreptheorem}[2]{%
	\newenvironment{rep#1}[1]{%
		\def\rep@title{#2 \ref{##1}}%
		\begin{rep@theorem}}%
		{\end{rep@theorem}}}
\makeatother

\theoremstyle{plain}
\newtheorem{thm}{Theorem}[section]
\newreptheorem{theorem}{Theorem}
\newtheorem{lem}[thm]{Lemma}
\newtheorem{cor}[thm]{Corollary}
\newtheorem{prop}[thm]{Proposition}

\theoremstyle{definition}
\newtheorem{ex}[thm]{Example}
\newtheorem{rem}[thm]{Remark}
\theoremstyle{definition}
\newtheorem{defi}[thm]{Definition}

\newtheorem{thm*}{Theorem}[section]
\newtheorem{prob*}[thm*]{Problem}

\DeclareMathOperator{\sln}{sl_{n}}
\DeclareMathOperator{\sl3}{sl_{3}}
\DeclareMathOperator{\R}{\mathbb{R}}
\DeclareMathOperator{\N}{\mathbb{N}}

\DeclareMathOperator{\Z}{\mathbb{Z}}

\DeclareMathOperator{\C}{\mathbb{C}}

\DeclareMathOperator{\sgn}{sgn}
\DeclareMathOperator{\wt}{wt}
\DeclareMathOperator{\str}{str}

\newcommand{\riga}{\Gamma^{*}_a}

\newcommand{\lsp}{Littelmann--Berenstein-Zelevinsky }

\DeclareMathOperator{\ii}{\mathbf{i}}
\DeclareMathOperator{\jj}{\mathbf{j}}

\newcommand{\PPsi}{\Lambda}

\DeclareMathOperator{\RR}{\Phi}

\newcommand{\W}{\mathcal{W}(w_0)}
\DeclareMathOperator{\w}{\mathcal{W}(w)}

\begin{document}

\title[crossing formula for string parametrizations]{Combinatorics of canonical bases revisited: String data in type $A$}

\author{Volker Genz}
\address{Mathematical Institute, Ruhr-University Bochum}
\email{volker.genz@gmail.com}

\author{Gleb Koshevoy}
\address{IITP Russian Academy of Sciences, MCCME and Poncelet Center}
\email{koshevoyga@gmail.com}

\author{Bea Schumann}
\address{Mathematical Institute, University of Cologne}
\email{bschumann@math.uni-koeln.de}

\begin{abstract} We give a formula for the crystal structure on the integer points of the string polytopes and the $*$-crystal structure on the integer points of the string cones of type $A$ for arbitrary reduced words. As a byproduct we obtain defining inequalities for Nakashima-Zelevinsky string polytopes. Furthermore, we give an explicit description of the Kashiwara $*$-involution on string data for a special choice of reduced word.
\end{abstract}

\maketitle

\section*{Introduction}
Let $\mathfrak{g}$ be a simple complex Lie algebra of rank $n-1$ and $V$ a finite dimensional representation of $\mathfrak{g}$. Much information of $V$ is encoded in a directed graph with arrows colored by $\{1,2,\ldots,n-1\}$, called the crystal graph of $V$ \cite{Ka91}. For instance, this crystal graph is connected if and only if $V$ is irreducible, the character of $V$ is encoded in the vertices of the crystals graph and there exists a simple notion of the tensor product of two crystal graphs yielding the crystal graph of the tensor product of two representations.

For $V$ irreducible, its crystal graph has a unique source corresponding to a highest weight vector of $V$. Making use of this fact, Littelmann \cite{Lit} and Berenstein-Zelevinsky \cite{BZ, BZ2} gave a bijection between the vertices of this graph as integer points of a rational convex polytope, called the \lsp string polytope. 

The rule for assigning an integer point in the \lsp string polytope to a vertex $v$ is as follows. 
Let $x_1$ be the largest integer such that there are $x_1$ consecutive arrows of color $i_1$ ending in $v$. Let $v_1$ be the source of this sequence of arrows.  
Let $x_2$ be the length of the longest sequence of arrows of a color $i_2$ ending in $v_1$ and so on. If we pick the colors $i_1,i_2,\ldots,i_N$ according to the appearance in a reduced decomposition of the longest Weyl group element of $\mathfrak{g}$, this procedure ends at the source of the graph. Then the vertex $v$ maps to the integer point $(x_1,x_2,\ldots,x_N)\in \mathbb{N}^N$, called the string datum of $v$.

\lsp string polytopes have a vast amount of applications. They are generalizations of Gelfand-Tsetlin polytopes (\cite{Lit}), appear as Newton-Okounkov bodies for flag varieties (\cite{FFL1,K15}) and in Gross-Hacking-Keel-Kontsevich's construction of canonical bases for cluster varieties (\cite{BF,GKS2}).

We consider the following problem for the string polytope of an irreducible representation $V$ associated to the reduced word $\ii=(i_1,i_2,\ldots,i_N)$ of the longest Weyl group element of $\mathfrak{g}$.
\begin{prob*}\label{proplem}
	 Give a formula for the operator ${f}_a$ on the integer points of the string polytope $P$ defined as follows. For two integer points $x$ and $x'$ in $P$ we have $f_a x =x'$, if the corresponding vertices $v$ and $v'$ in the crystal graph are connected by an arrow of color $a$.
\end{prob*}

Problem \ref{proplem} is easy to solve for $a=i_1$. In this case we have 
$${f}_a(x_1,x_2,\ldots,x_N)=(x_1+1,x_2,\ldots,x_N).$$
There is, however, no obvious solution for arbitrary $a$. For $\sl3(\C)$ and the reduced word $s_1s_2s_1$, one can deduce from an explicit construction of the crystal graph (\cite{DKKA2})
that $f_2(x_1, x_2, x_3)$ is equal to  $(x_1, x_2+1, x_3)$ if $x_1\le x_2-x_3$ and $(x_1-1, x_2+1, x_3+1)$ otherwise. In this work we solve Problem \ref{proplem} by establishing a formula for the operator ${f}_a$ for any $a$ in the case that $\mathfrak{g}=\sln(\mathbb{C})$.

For $a\in \{1,2,\ldots,n-1\}$ and a reduced word $\ii=(i_1,i_2,\ldots,i_N)$ of the longest element of the Weyl group of $\sln(\mathbb{C})$ we define in Section \ref{wiringandrein} finitely many sequences $\gamma=(\gamma_j)$ of positive roots of $\sln(\mathbb{C})$ with certain properties which we call \emph{$a$-crossings}. These sequences come with an order relation $\preceq$. We further introduce maps $r$, $s$ associating to $\gamma$ the vectors $r(\gamma)$, $s(\gamma)\in \mathbb{Z}^N$.

Our main result reads as follows, where $\langle \cdot,\cdot \rangle$ is the standard scalar product on $\mathbb{Z}^N$.
\begin{reptheorem}{crossing}
Let $\gamma$ be minimal such that $\langle x,  r(\gamma) \rangle$ is maximal. Then $${f}_{a} x=x+s(\gamma).$$
\end{reptheorem}
Theorem \ref{crossing} is in analogy to the Crossing Formula established in \cite[Theorem 2.13, Proposition 2.20]{GKS1}, which computes the operator ${f}_a$ on the polytopes arising from Lusztig's parametrizations of the crystal graph. Indeed, the two formulae may be viewed as dual since the roles of maximum and minimum and the vectors $r(\gamma)$, $s(\gamma)$ interchange. We elaborate on this duality in \cite{GKS3}.

Theorem \ref{crossing} gives rise to two applications. The Verma module of $\mathfrak{g}$ of weight $0$ has a crystal graph $B(\infty)$ with a unique source. 
Kashiwara \cite{Ka2} defined an involution $*$ on the vertices of $B(\infty)$, leading to a second crystal graph $B(\infty)^*$ with the same set of vertices. Namely, there is an arrow from $v_1$ to $v_2$ of color $a$ in $B(\infty)^*$ if and only if there is an arrow from $v_1^*$ to $v_2^*$ of color $a$ in $B(\infty)$.

Associating integer vectors to the vertices of $B(\infty)^*$ by taking their string data, we obtain a rational polyhedral cone called the string cone \cite{Lit,BZ,BZ2} which contains the \lsp string polytope. 

A variation of Problem \ref{proplem} now arises, replacing the \lsp string polytope by the string cone and the crystal graph of an irreducible representation by $B(\infty)^*$. In Theorem \ref{stringinfty} we provide a solution to this problem for $\mathfrak{g}=\sln$. Indeed the crystal graph of each irreducible representation $V$ is a full subgraph of $B(\infty)^*$. Making use of this fact we deduce Theorem \ref{stringinfty} from Theorem \ref{crossing}.

A second crystal graph for the irreducible representation $V$ is obtained as a full subgraph of $B(\infty)$. The set of corresponding string parameters is, due to a result of Fujita-Naito \cite{FN}, again the set of integer points in a rational polytope, called the Nakashima-Zelevinsky string polytope. These polytopes have been found to coincide with Newton-Okounkov bodies for flag varieties \cite{FN,FO}. They also appear in \cite{CFL} among Newton-Okounkov bodies inducing semitoric degenerations of Schubert varieties associated to maximal chains in the corresponding Bruhat graphs.

For Nakashima-Zelevinsky polytopes problem \ref{proplem} has been solved in the work of Kashiwara \cite{Ka2} and Nakashima-Zelevinsky \cite{NZ,N99}. It is, however, a difficult problem to compute the inequalities which cut the Nakashima-Zelevinsky polytopes out of the string cone. This is so far only known in a few special cases \cite{N99,H}. Using Theorem \ref{stringinfty} we obtain these inequalities for all reduced words of the longest Weyl group element of $\sln$ in Theorem \ref{inequpol}.

The paper is organized as follows. In Section \ref{crystals} we recall the background on crystals. In Section \ref{symgroup} we recall facts about reduced words for elements of the symmetric group. In Section \ref{sectionstring} string cones and \lsp string polytopes, as well as their crystal structures, are discussed. 

In Section \ref{wiringandrein} we introduce the main combinatorial tools of this paper, namely the notion of wiring diagrams and Reineke crossings. The main result (Theorem \ref{crossing}), providing a formula for the crystal structure on \lsp string polytopes, is stated in Section \ref{sec:mainthm}. We further prove the Dual Crossing Formula for the $*$-crystal structure on the string cone in this section. 

In Section \ref{sec:NZ} Nakashima-Zelevinsky string polytopes are introduced and their defining inequalities are computed.

Section \ref{sec:Lus} deals with Lusztig's parametrization of the canonical basis and recalls facts from \cite{GKS1} which are used in the proof of Theorem \ref{crossing} which is presented in Section \ref{mainproof}.

In Section \ref{sec:kashstar} we give a description of the piecewise linear Kashiwara $*$-involution on string data. In particular, we obtain a linear isomorphism between the \lsp polytope and the Nakashima-Zelevinsky polytope for a specific reduced word.

\section*{Acknowledgement}
V. Genz and G. Koshevoy were partially supported by the SFB/TRR 191. V. Genz would like to thank the Independent University of Moscow and the Labaratoire J.-V. Poncelet for their hospitality. G. Koshevoy was supported by the grant RSF 16-11- 10075. He would furthermore like to thank the University of Cologne and the Ruhr-Univerity Bochum for their hospitality. B. Schumann was supported by the SFB/TRR 191. B. Schumann would furthermore like to thank Xin Fang, Peter Littelmann, Valentin Rappel, Christian Steinert and Shmuel Zelikson for helpful discussions.

\section{Crystals}\label{crystals}

\subsection{Notation}\label{notation}
Let $\N=\{0,1,2,\ldots\}$ be the natural numbers and $\mathfrak{g}=\sln(\mathbb{C})$, $\mathfrak{h}\subset \mathfrak{g}$ its Cartan subalgebra consisting of the diagonal matrices in $\mathfrak{g}$. We abbreviate 
$$[n]:=\{1,2,\ldots, n\}$$
 and define for $k\in [n]$ the function $\epsilon_k\in \mathfrak{h}^*$ by $\epsilon_k(\text{diag}(h_1,h_2,\ldots,h_n))=h_k$. We denote by $\Phi^+$ the set of positive roots of $\mathfrak{g}$ given by
\begin{equation*}
\Phi^+=\{\alpha_{k,\ell}=\epsilon_k-\epsilon_{\ell} \mid 1\leq k<\ell \leq n\}.
\end{equation*}
For $a\in [n-1]$, the simple root $\alpha_a$ of $\mathfrak{g}$ is given by $\alpha_a=\alpha_{a,a+1}=\epsilon_{a}-\epsilon_{a+1}$. We denote by $N=\frac{n(n-1)}{2}$ the cardinality of $\Phi^+$.

To $a\in [n-1]$ we associate the fundamental weight $\omega_a=\sum_{s\in [a]}\epsilon_s$ of $\mathfrak{g}$. 
Let $P\subset \mathfrak{h}^*$ (resp. $P^+\subset \mathfrak{h}^*$) the $\mathbb{Z}$-span (resp. $\mathbb{Z}_{\ge 0}$-span) of the set of fundamental weights $\{\omega_a\}_{a\in [n-1]}$ of $\sln(\mathbb{C})$. We call $P$ the weight lattice and $P^+$ the set of dominant integral weights.

Let $U_q(\sln)$ be the $\mathbb{Q}(q)$- algebra with generators $E_a,F_a,K_a^{\pm 1}$, $a\in [n-1]$ and the following relations for $b\in [n-1]\setminus \{a\}$
$$K_aK_a^{-1}=K_a^{-1}K_a=1, \quad K_aK_b=K_bK_a, \quad K_aE_aK_a^{-1}=q^{2}E_a$$
$$K_aF_aK_a^{-1}=q^{2}F_a, \quad E_aF_b-F_bE_a=0, \quad E_aF_a-F_aE_a=\frac{K_a-K_a^{-1}}{q-q^{-1}}$$
\begin{align*} \text{If }b= a\pm 1: \ &  E_a^{2}E_b+E_bE_a^2=(q+q^{-1})E_aE_bE_a,\\ & F_a^{2}F_b+F_bF_a^2=(q+q^{-1})F_aF_bF_a, \\
& K_aE_bK_a^{-1}=q^{-1}E_b, \quad K_aF_bK_a^{-1}=qF_b.
\end{align*}
\begin{align*} \text{If }b\neq a\pm 1: \ & E_aE_b=E_bE_a, \quad F_aF_b=F_bF_a,\\ & K_aE_bK_a^{-1}=E_b, \quad K_aF_bK_a^{-1}=F_b.
\end{align*}

For $m\in \N$, let $[m]_{q}:=q^{m-1}+q^{m-3}+\ldots +q^{-m+1}$. For $x\in U_q(\sln)$ we set
\begin{equation}\label{dividedpower}
x^{(m)}:=\frac{x^m}{([m]_{q}[m-1]_q\cdots [2]_q)}.
\end{equation}

For $\lambda\in P^+$ we denote by $V(\lambda)$ the irreducible $U_q(\sln)$-module of highest weight $\lambda$.

We finally denote by $U_q^-\subset U_q(\text{sl}_n)$ be the subalgebra generated by $\{F_a\}_{a\in[n-1]}$.

\subsection{Crystals}
We recall the definition of crystals from \cite[Section 7]{Ka3}
\begin{defi} \label{abstrcryst} A \emph{crystal} $B$ is a set endowed with the following maps.
\begin{align*}
\wt: B &\rightarrow P, &
\varepsilon_a: B &\rightarrow \mathbb{Z}\sqcup \{-\infty\}, \quad  \varphi_a:B \rightarrow  \mathbb{Z}\sqcup \{-\infty\}, \\
{e}_a: B &\rightarrow B \sqcup \{0\}, \quad &{f}_a:B&\rightarrow B \sqcup \{0\} \quad \text{ for }a\in [n-1].
\end{align*}
Here $0$ is an element not included in $B$. The above maps satisfy the following axioms for $a\in [n-1]$ and $b, b'\in B$
\begin{itemize}
\item[(C1)] $\varphi_a(b)=\varepsilon_a(b)+\wt(b)(h_a)$,
\item[(C2)] if $b\in B$ satisfies ${e}_a b \ne 0$ then $$\wt({e}_a b)=\wt(b) + \alpha_a, \quad \varphi_a({e}_a b)=\varphi_a(b)+1,\quad \varepsilon_a({e}_a b)=\varepsilon_a(b)-1,$$
\item[(C3)] if $b\in B$ satisfies ${f}_a b \ne 0$ then $$\wt({f}_a b)=\wt(b) - \alpha_a, \quad \varphi_a({f}_a b)=\varphi_a(b)-1, \quad \varepsilon_a({f}_a b)=\varepsilon_a(b)+1,$$
\item[(C4)] ${e}_a b =b'$ if and only if ${f}_a b' =b$,
\item[(C5)] if $\varepsilon_a b = - \infty$, then ${e}_a b={f}_a b =0$.
\end{itemize}
Here we put $-\infty+k=-\infty$ for $k\in \mathbb{Z}$.
\end{defi}

Let $B_1$ and $B_2$ be crystals. A map $\PPsi:B_1 \sqcup \{0\} \rightarrow B_2\sqcup \{0\}$ satisfying $\PPsi(0)=0$ is called a \emph{strict morphism of crystals} if $\PPsi$ commutes with all ${f}_a$, ${e}_a$ ($a\in [n-1]$) and if for $b\in B_1$, $\PPsi(b)\in B_2$ we have
\begin{equation*}
\wt(\PPsi(b)) =\wt(b), \quad \varepsilon_a(\PPsi(b))=\varepsilon_a(b), \quad \varphi_a(\PPsi(b))=\varphi_a(b)
\end{equation*}
for all $a\in [n-1]$. An injective strict morphism is called a \emph{strict embedding of crystals} and a bijective strict morphism is called an \emph{isomorphism of crystals}.

\begin{defi}\label{crystens} Let $B_1$ and $B_2$ be crystals. The set 
$$B_1\otimes B_2:=\{b_1 \otimes b_2 \mid b_1 \in B_2, \ b_2\in B_2\}$$ equipped with the following crystal structure is called the \emph{tensor product} of $B_1$ and $B_2$.
For $a\in [n-1]$
\begin{align*}
\wt(b_1 \otimes b_2)&=\wt(b_1)+\wt(b_2),\\
\varepsilon_a(b_1\otimes b_2)&= \max\{\varepsilon_a(b_1),\varepsilon_a(b_2)-\wt(b_1)(h_a)\}, \\
\varphi_a(b_1\otimes b_2)&= \max\{\varphi_a(b_2),\varphi_a(b_1)-\wt(b_2)(h_a)\}, \\
{e}_a(b_1\otimes b_2) & = \begin{cases} {e}_a b_1 \otimes b_2 & \text{ if } \varphi_a(b_1) \ge \varepsilon_a(b_2) \\ b_1 \otimes {e}_a b_2 & \text{ else,} \end{cases} \\
{f}_a(b_1\otimes b_2) & = \begin{cases} {f}_a b_1 \otimes b_2 & \text{ if } \varphi_a(b_1) > \varepsilon_a(b_2) \\ b_1 \otimes {f}_a b_2 & \text{ else.} \end{cases}
\end{align*}
\end{defi}

\subsection{Crystals of representations}\label{sec:crysrep}
We recall the \emph{crystal bases} $B(\infty)$ and $B(\lambda)$ of $U_q^{-}$ and $V(\lambda)$, respectively, from \cite[Sections 2 and 3]{Ka91}.

Let $a\in [n-1]$. For $P\in U_q^{-}$ there exist unique $Q,R \in U_q^{-}$ such that 
$$E_aP-PE_a=QK_a+RK_a^{-1}.$$ 
We define $e'_a(P)=R$.
As vector spaces, we have
$$U_q^{-}=\displaystyle\bigoplus_{m\ge 0} F_a^{(m)}\ker(e'_a).$$ 
We define the \emph{Kashiwara operators} ${e}_a$, ${f}_a$ on $U_q^-$ for $u\in \ker(e'_a)$ by
\begin{equation}
\label{kashops}
{f}_a(F_a^{(m)}u)= F_a^{(m+1)}u, \quad {e}_a(F_a^{(m)}u)= F_a^{(m-1)}u.
\end{equation}

Let $A$ be the subring of $\mathbb{Q}(q)$ consisting of rational functions $g(q)$ without a pole at $q=0$. Let $\mathcal{L}(\infty)$ be the $A$-lattice generated by all elements of the form
\begin{equation}\label{seqKas}
{f}_{i_1}{f}_{i_2}\cdots{f}_{i_{\ell}}(1)
\end{equation}
and let $B(\infty)\subset \mathcal{L}(\infty)/q\mathcal{L}(\infty)$ be the subsets of all residues of elements of the form \eqref{seqKas}.

For $b\in B(\infty)$ let $\wt(b)$ be the weight of the corresponding element in $U_q^-$.		 For $a\in [n-1]$ we furthermore set $\varepsilon_a(b)=\max\{{e}_a^k \ne 0 \mid k\in \mathbb{N}\}.$ This endows $B(\infty)$ with the structure of an crystal (see Definition \ref{abstrcryst}).

We let $*:U_q^- \rightarrow U_q^-$ be the $\mathbb{Q}(q)$-anti-automorphism of $U_q^-$ such that $E_a^*=E_a$ for all $a\in [n-1]$. By \cite[Theorem 2.1.1]{Ka2} we have $B(\infty)^{*}=B(\infty)$. Clearly $*$ preserves the function $\wt$. We denote by ${f}_a^*(x)=({f}_ax^*)^*$, ${e}_a^*(x)=({e}_ax^*)^*$ and $\varepsilon_a^*(x)=\varepsilon_a(x^*)$ the $*$-twisted maps. This endows $B(\infty)$ with a second structure of a crystal. We denote the crystal given by the set $B(\infty)$ and the twisted maps by $B(\infty)^*$. By construction $*$ induces a crystal isomorphism between $B(\infty)$ and $B(\infty)^*$.

For $\lambda\in P^+$ let $\pi_{\lambda}:U_q^- \rightarrow V(\lambda)$ be the surjection $u\mapsto u v_\lambda$, where $v_{\lambda}$ is a highest weight vector of $V(\lambda)$. The operators $e_a$ and $f_a$ defined in \eqref{kashops} descend to $V(\lambda)$ and we denote by $\mathcal{L}(\lambda)$ the $A$-lattice generated by all elements of the form 
\begin{equation}\label{seqKaslam}
{f}_{i_1}{f}_{i_2}\cdots{f}_{i_{\ell}}(v_{\lambda})
\end{equation}
and by $B(\lambda)\subset \mathcal{L}(\lambda)/q\mathcal{L}(\lambda)$ the subsets of all residues of elements of the form \eqref{seqKaslam}. 

For $b\in B(\lambda)$ let $\wt(b)$ be the weight of the corresponding element in $V(\lambda)$. For $a\in [n-1]$ we furthermore set 
\begin{align*}
\varepsilon_a(b)&=\max\left\{{e}_a^k b \ne 0 \,\middle\mid\, k\in \mathbb{N}\right\},\\ \varphi_a(b)&=\max\left\{{f}_a^k b \ne 0 \,\middle\mid\, k\in \mathbb{N}\right\}.
\end{align*}
This endows $B(\lambda)$ with the structure of a crystal (see Definition \ref{abstrcryst}). 

We embed $B(\lambda)$ into $B(\infty)$ with accordingly shifted weight as follows.

By \cite[Theorem 4]{Ka91} we have $\pi_{\lambda}(\mathcal{L}(\infty))=\mathcal{L}(\lambda)$ inducing a map $\overline{\pi}_{\lambda}:\mathcal{L}(\infty)/q\mathcal{L}(\infty) \rightarrow \mathcal{L}(\lambda)/q\mathcal{L}(\lambda)$ with the following properties:
\begin{itemize}
\item ${f}_a\circ \overline{\pi}_{\lambda}=\overline{\pi}_{\lambda} \circ {f}_a$ for all $a\in [n-1]$,
\item If $\overline{\pi}_{\lambda}(b)\ne 0$ we have ${e}_a\overline{\pi}_{\lambda}(b)=\overline{\pi}_{\lambda}({e}_a b)$ for all $a\in [n-1]$,
\item $\overline{\pi}_{\lambda}:B(\infty) \setminus \{\overline{\pi}_{\lambda}^{-1}(0)\} \rightarrow B(\lambda)$ is bijective.
\end{itemize}

For $\lambda \in P$ an integral weight, let $R_{\lambda}=\{r_{\lambda}\}$ be the crystal consisting of one element satisfying $\wt(r_{\lambda})=\lambda$, $\varepsilon_a(r_{\lambda})=-\lambda(h_a), \
\varphi_a(r_{\lambda})=0$ and ${e}_a r_{\lambda}={f}_a r_{\lambda}=0$ for all $a\in [n-1].$

By \cite[Corollary 5.3.13]{J95}, \cite[Theorem 3.1]{N99} $$\widetilde{B}(\lambda):=\{b\otimes r_{\lambda} \in B(\infty) \otimes R_{\lambda} \mid \overline{\pi}_{\lambda}(b\otimes r_{\lambda})\ne 0\}$$ is a subcrystal of $B(\infty)\otimes R_{\lambda}$ and $\overline{\pi}_{\lambda}$ induces an isomorphism of crystals $\widetilde{B}(\lambda)\cong B(\lambda)$. Furthermore,
\begin{equation}\label{blambdainbinfty}
\widetilde{B}(\lambda)=\{b\otimes r_{\lambda} \in B(\infty) \otimes R_{\lambda} \mid \varepsilon^*_a(b) \le \lambda(h_a) \ \forall a \in [n-1]\}\cong B(\lambda).
\end{equation}

\section{Symmetric groups, reduced words and wiring diagrams}\label{symgroup}
\subsection{Symmetric groups and reduced words}
Let $\mathfrak{S}_n$ be the symmetric group in $n$ letters. The group $\mathfrak{S}_n$ is generated by the simple transpositions $\sigma_a$ ($a \in [n-1]$) interchanging $a$ and $a+1$.

A reduced expression of $w\in \mathfrak{S}_n$ is a decomposition of $w$
$$w=\sigma_{i_1}\sigma_{i_2}\cdots \sigma_{i_k}$$
into a product of simple transposition with a minimal possible number of factors. We call $k$ the \emph{length $\ell(w)$ of $w$}. For a reduced expression of $w\in \mathfrak{S}_n$ we write $\ii:=(i_1,i_2,\ldots, i_N)$ and call $\ii$ a \emph{reduced word} (for $w$). The set of reduced words for $w$ is denoted by $\mathcal{W}(w)$.

The group $\mathfrak{S}_n$ has a unique longest element $w_0$ of length $N:=\frac{n(n-1)}{2}$. 
We have two operations on the set of reduced words $\W$.
\begin{defi}\label{moves}
A reduced word $\jj=(j_1,\ldots,j_N)\in \W$ is said to be obtained from $\ii=(\ii_1,i_k,i_{k+1},\ii_2)\in \W$ by a \emph{$2$-move at position $k\in [N-1]$} if $\jj=(\ii_1,i_{k+1},i_k,\ii_2)$ and $|i_k-i_{k+1}|>1$.

A reduced word $\jj=(j_1,\ldots,j_N)$ is said to be obtained from $$\ii=(\ii_1,i_k,i_{k+1},i_{k+2},\ii_2)\in \W$$ by a \emph{$3$-move at position $k\in [N-1]$} if $i_{k}=i_{k+2}$, $\jj=(\ii_1,i_{k+1},i_k,i_{k+1},\ii_2)$ and $|i_k-i_{k+1}|=1.$
\end{defi}

A pair $(p,q)\in [n]^2$ with $p<q$ is called an \emph{inversion for $w\in \mathfrak{S}_n$} if $w(p)>w(q)$. Let $I(w)$ be the set of inversions for $w\in \mathfrak{S}_n$. We have the $|I(w)|$ is equal to the length $\ell(w)$.

A total ordering $<$ on $I(w)$ is called a \emph{reflection ordering} if for any triple $(p,q),(p,r),(q,r)\in I(w)$ we either have $(p,q)<(p,r)<(q,r)$ or $(q,r)<(p,r)<(p,q)$.

It is well known that the sets $\w$ and $I(w)$ are in natural bijection (see e.g. \cite[Proposition 2.13]{D93}). Under this bijection the reflection order corresponding to $\ii=(i_1,\ldots,i_k)\in \w$ is given by
$$(p_1,q_1)<\ldots <(p_k,q_k)$$
where $p_j=\sigma_{i_1}\cdots\sigma_{i_{j-1}}(i_j)$, $q_j=\sigma_{i_1}\cdots\sigma_{i_{j-1}}(i_{j}+1)$.

\begin{rem}\label{convexorder}
Let $\ii \in \W$. The set $I(w_0)$ is in bijection with $\Phi^+$ via the map
\begin{equation}\label{vertexroot}
(p,q) \mapsto \alpha_{p,q},
\end{equation}
where $\alpha_{p,q}$ is defined in Section \ref{notation}. The reflection order corresponding to $\ii$ induces a total ordering on $\Phi^+$ in this case.  
\end{rem}

\section{String parametrizations}\label{sectionstring}

\subsection{String parametrization}
\subsubsection{Kashiwara embedding and string parameters}\label{strinpara}
Let $\ii \in \W$ and $b\in B(\infty)$. For $1\le k \le N$ we recursively define $$x_k=\varepsilon_{i_k}\left({e}_{i_{k-1}}^{x_{k-1}}\cdots {e}_{i_1}^{x_1}b\right)$$ and call $\str_{\ii}(b):=(x_1,\ldots, x_N)$ the \emph{string datum of $b$ in direction $\ii$}.

By \cite[Lemma 5.3]{Lit2} we have
\begin{equation}\label{hw}
{e}_{i_N}^{x_N}\,\cdots\, {e}_{i_1}^{x_1}b=b_{\infty},
\end{equation}
where $b_\infty$ is the element in $B(\infty)$ of highest weight.

By \eqref{hw} the map $\str_{\ii}$ is injective. We denote by ${\mathcal{S}_{\ii}}=\str_{\ii}(B(\infty))$ the image of $\str_{\ii}$. Let $\mathcal{S}_{\ii}^{\mathbb{R}}\subset \mathbb{R}^{N}$ be the cone spanned by ${\mathcal{S}_{\ii}}$. By \cite[Proposition 1.5]{Lit}, \cite[Proposition 3.5]{BZ2} $\mathcal{S}_{\ii}^{\mathbb{R}}$ is a rational polyhedral cone, called the \emph{string cone}, and ${\mathcal{S}_{\ii}}$ are the integral points of $\mathcal{S}_{\ii}^{\mathbb{R}}$.

Recall the definition of $\varepsilon^*_a$ and ${e}^*_a$ from Section \ref{sec:crysrep}. Now let 
$$x_k=\varepsilon^*_{i_k}\left(({e}_{i_{k-1}}^{*})^{x_{k-1}}\,\cdots\, ({e}_{i_{1}}^{*})^{x_{1}}b\right).$$ 
We call $\str^*_{\ii}(b):=(x_1,\ldots,x_N)$ the \emph{$*$-string datum of $b$ in direction $\ii$}.

\begin{lem}\label{lem:str} For $b\in B(\infty)$ we have
\begin{align} \label{stringcom}
\str_{\ii}(b^*)&=\str^*_{\ii}(b) \\ \label{starinj}
b_{\infty}&=({e}^*_{i_N})^{x_N}\,\cdots\, ({e}^*_{i_1})^{x_1}b, \\ \label{string2}
\mathcal{S}_{\ii}&=\str^*_{\ii}(B(\infty)).
\end{align}
\end{lem}

\begin{proof}
Let $(y_1, \,\dots, y_N)=\str_{\ii}(b^*)$. By \eqref{hw} we have ${e}_{i_N}^{y_N} \,\cdots\, {e}_{i_1}^{y_1}b^*=b_{\infty}$. Applying $*$ to both sides and using that ${e}^*_a b=({e}_a b^*)^*$ and $b_{\infty}^*=b_{\infty}$ we get
$$({e}_{i_{N}}^{*})^{y_{N}}\,\cdots\,({e}_{i_{1}}^{*})^{y_{1}}=b_{\infty}.$$
Since
\begin{align*}
y_k=\varepsilon^*_{i_k}\left(({e}_{i_{k-1}}^{*})^{y_{k-1}}\,\cdots\, ({e}_{i_{1}}^{*})^{y_{1}}b\right)&=\varepsilon_{i_k}\left(({e}_{i_{k-1}}^{*})^{y_{k-1}}
\,\cdots\, ({{e}^*}_{i_1})^{y_1}b\right)^* \\
&= \varepsilon_{i_k}\left({{e}}_{i_{k-1}}^{y_{k-1}}\,\cdots\, {{e}}_{i_1}^{y_1}b^*\right)=(\str_{\ii}(b^*))_k
\end{align*}
we obtain \eqref{stringcom}. Now \eqref{starinj} follows by applying \eqref{hw} to $\str_{\ii}(b^*)$.

Since the crystals $B(\infty)$ and $B(\infty)^*$ have the same underlying set (see Section \ref{sec:crysrep}), Equation \eqref{string2} follows from \eqref{stringcom}.
\end{proof}

\subsection{Crystal structures on string data}\label{crysonstring}
In this section we equip $S_{\ii}$ with two crystal structures isomorphic to $B(\infty)$.

For $a\in [n-1]$ and $k\in\Z$ let $b_a(k)$ be a formal symbol. We denote by $B_a:=\{b_a(k) \mid k \in \mathbb{Z}\}$ the crystal, such that for $a'\in[n-1]$
\begin{align*}
\varepsilon_{a'}(b_{a}(k))&= \varphi_{a'}(b_{a}(-k))= \begin{cases}
-k, &\text{if $a=a'$,}\\-\infty,&\text{else,}\end{cases}\\
\wt(b_a(k))&=k\alpha_a,\\
{f}_{a'}(b_{a}(k))&=\begin{cases} b_{a}(k-1), &\text{if $a'=a$,}\\0&\text{else,}\end{cases}\\
{e}_{a'}(b_{a}(k))&=\begin{cases} b_a(k+1), &\text{if $a'=a$,}\\0&\text{else.}\end{cases}
\end{align*}

By \cite[Theorem 2.2.1]{Ka2} there exists for any $a\in [n-1]$ a unique strict embedding of crystals given by 
\begin{align}\label{Kashiwaraem}
\PPsi_a: B(\infty) & \hookrightarrow B(\infty) \otimes B_a \\ \notag
b_{\infty} & \mapsto b_{\infty}\otimes b_a(0).
\end{align}
In \cite[Theorem 2.2.1 and its proof]{Ka2} (see also  \cite[Section 2.4]{NZ}) the following statement is proved.

\begin{lem}\label{imemb} Let $b\in B(\infty)$ and $m=\varepsilon^*_a(b)$. We have
	\begin{equation*}
	\PPsi_a(b)=\left({e}^*_a\right)^m b \otimes b_a(-m).
	\end{equation*}
\end{lem}
Lemma \ref{imemb} naturally provides two crystal structures on $\mathcal{S}_{\ii}$ as follows. 

Let $\ii=(i_1,\ldots,i_N)\in \W$. We iterate the map \eqref{Kashiwaraem} along $\ii$ by setting
$$\PPsi_{\ii}=\PPsi_{i_N} \circ \PPsi_{i_{N-1}} \circ \ldots \circ\PPsi_{i_1}.$$
Combining Lemma \ref{lem:str} with Lemma \ref{imemb} 
 we obtain the strict embedding 
\begin{equation*}
\PPsi_{\ii}(b)=b_{\infty} \otimes b_{i_1}(-x_1) \otimes b_{i_2}(-x_2) \cdots \otimes b_{i_N}(-x_N),
\end{equation*}
where $(x_1,x_2,\ldots,x_N)=\str^*_{\ii}(b)=\str_{\ii}(b^*)$. Identifying $\mathcal{S}_{\ii}$ with $\PPsi_{\ii}(B(\infty))$ via 
$$(x_1,\ldots,x_N) \mapsto b_{\infty} \otimes b_{i_N}(-x_N) \otimes \cdots \otimes b_{i_1}(-x_1).$$
yields two crystal structures $B(\infty)$ and $B(\infty)^*$ on $\mathcal{S}_{\ii}$.

From $\PPsi_{\ii}(B(\infty))\subset \{b_\infty\} \otimes B_{i_1} \otimes  \ldots  \otimes B_{i_N}$ we obtain the following explicit description of  the crystal structure on $\mathcal{S}_{\ii}$ resulting from $B(\infty)$.
Let $(c_{i,j})$ be the Cartan matrix of $\text{sl}_n(\mathbb{C})$. For $k\in[N]$ and $x\in\mathcal{S}_{\ii}$ we set 	
\begin{align}\label{strstat1}\eta_k(x):&=x_k+ \sum_{ k <\ell\le N}c_{i_k,i_\ell}x_\ell.\end{align}
\begin{lem}\label{strBiso}
The crystal structure on  $\mathcal{S}_{\ii}$ obtained from $B(\infty)$ via the bijection $b\mapsto \str^*_{\ii}(b)$ is given as follows.
For $x\in \mathcal{S}_{\ii}$ and $a\in[n-1]$
\begin{align}\begin{split}\label{scrys}
\varepsilon_a(x)&= \max\left\{\eta_k(x) \mid k\in[N],\, i_k=a \right\},\qquad \wt(x)=-\displaystyle\sum _{k=1}^{N} x_k \alpha_{i_k},\\
{f}_a (x) & = x+\left(\delta_{k,\ell^{x}}\right)_{k\in[N]}\\ {e}_a (x) &=\begin{cases}x-\left(\delta_{k, \ell_{x}}\right)_{k\in[N]} & \text{if }\varepsilon_a(x)>0, \\ 0, &\text{else,}\end{cases}\end{split}
\end{align}
where $\ell^{x}\in[N]$ is minimal with $i_{\ell^x}=a$ and $\eta_{\ell^x}(x)=\varepsilon_a(x)$ and where
$\ell_{x}\in[N]$ is maximal with $i_{\ell_x}=a$ and $\eta_{\ell_x}(x)=\varepsilon_a(x)$.	
\end{lem}

The crystal structure on  $\mathcal{S}_{\ii}$ obtained from $B(\infty)^*$ via the bijection $b\mapsto \str^*_{\ii}(b)$ is given as follows. 

By \cite[Proposition 2.3]{Lit} (see also \cite[Theorem 2.7]{BZ}) we introduce piecewise linear bijections $\Psi^{\ii}_{\jj} : \mathcal{S}_{\ii}^{\mathbb{R}} \rightarrow \mathcal{S}_{\jj}^{\mathbb{R}}$ between the string cones associated to reduced words $\ii, \jj\in \W$ satisfying for $b\in B(\infty)$ 
\begin{equation}\label{stringtrans}
\Psi^{\ii}_{\jj} \circ 
\str_{\ii}(b)=\str_{\jj}(b)
\end{equation}
as follows. If $\jj\in \W$ is obtained from $\ii\in \W$ by a $3$-move at position $k$ we set $y=\Psi_{\jj}^{\ii}(x)$ with 
	$$
	y = (x_1, \dots, x_{k-2},x'_{k-1} , x'_k, x'_{k+1} , x_{k+2}, \dots, x_N ),$$
		$$x'_{k-1}= \max(x_{k+1},x_k-x_{k-1}), \quad x'_k=x_{k+1}+x_{k-1} \text{ and }$$ 
		$$ x'_{k+1}=\min(x_{k-1},x_{k+1}).$$
	If $\jj\in \W$ is obtained from $\ii\in \W$ by a $2$-move at position $k$ we set
	$$\Psi_{\jj}^{\ii} \left(x_1, \dots, x_N \right) = \left( x_1, \dots, x_{k-1}, x_{k+1}, x_k, x_{k+2}, \dots, x_N\right).
	$$
	For arbitrary $\ii, \jj\in \W$ we define $\Psi_{\jj}^{\ii}:\mathcal{S}_{\ii} \rightarrow \mathcal{S}_{\jj}$ as the composition of the transition maps corresponding to a sequence of $2-$ and $3-$moves transforming $\ii$ into $\jj$.

\begin{lem}\label{eq:stringstarop}Let $x \in \mathcal{S}_{\ii},$ $a\in[n-1]$ and $\jj\in\W$ with $j_1=a$. Setting $y:=\Psi^{\ii}_{\jj}(x)\in\mathcal{S}_{\jj}$ we have
\begin{align}
\label{epsiastarstr}
\varepsilon_{a}^*(x)&=y_1,\qquad \wt(x)=-\displaystyle\sum _{k=1}^{N} x_k \alpha_{i_k},\\\notag
f_{a}^*(x) &= \Psi^{\jj}_{\ii}\left(y+(1,0,0,\dots)\right),\\\notag
e_{a}^*(x) &= \begin{cases}
\Psi^{\jj}_{\ii}\left(y-(1,0,0,\dots) \right), & \text{if $\varepsilon_a^*(x)>0,$}\\0,&\text{else.}
\end{cases}
\end{align}
\end{lem}
\begin{proof}
The statement follows from Lemma \ref{imemb} and \eqref{stringtrans}.
\end{proof}
In Theorem \ref{stringinfty} we give a formula for the crystal structure of Lemma \ref{eq:stringstarop}.

\subsection{String polytopes and their crystals structures}
Let $\lambda\in P^+$ and $\ii\in \W$. Recall from \eqref{blambdainbinfty} that the crystal $B(\lambda)$ is isomorphic to the subcrystal $\widetilde{B}(\lambda)$ of $B(\infty)\otimes R_{\lambda}$. Hence, using \eqref{hw} we get a bijection between $B(\lambda)$ and 
\begin{equation}
\label{strstr}
\mathcal{S}_{\ii}^{*}(\lambda):=\left\{\str_{\ii}(b) \mid b\otimes r_{\lambda}\in B(\infty)\otimes R_{\lambda}, \varepsilon^*_a(b) \le \lambda(h_a) \ \forall a \in [n-1]\right\}.\end{equation}

In \cite[Proposition 1.5]{Lit} it is shown that $\mathcal{S}_{\ii}^{*}(\lambda)$ is the set of integer points of the rational polytope \begin{equation}\label{litstring}
\mathcal{S}_{\ii}^{*}(\lambda)^{\mathbb{R}}=\left\{x \in \mathcal{S}_{\ii}^{\mathbb{R}} \,\middle\mid\, x_k+\displaystyle\sum_{k< \ell \le N}c_{i_k,i_\ell}x_k \le \lambda_{i_k}  \ \forall k \in [N]\right\}\subset \mathbb{R}^N.
\end{equation}
We call $\mathcal{S}_{\ii}^{*}(\lambda)^{\mathbb{R}}$ the \emph{\lsp string polytope}. 

By \eqref{strstr} we obtain the following crystal structure isomorphic to $B(\lambda)$ on $\mathcal{S}_{\ii}^{*}(\lambda) \subset \mathcal{S}_{\ii}$. Denoting by $\iota_{\lambda}: \mathcal{S}_{\ii}^{*}(\lambda) \hookrightarrow  \mathcal{S}_{\ii}$ the natural embedding we obtain
\begin{lem}\label{staronpoln}
For $x\in \mathcal{S}_{\ii}^{*}(\lambda)$ and $a\in[n-1]$ we have
\begin{align}
\notag\begin{split}
\varepsilon_a(x)& =\varepsilon^*_a(\iota_{\lambda}(x)), \qquad \wt(x)=\lambda+\wt(\iota_{\lambda}(x)),\qquad  \iota_{\lambda} {e}_a(x)= {e}^*_a\iota_{\lambda}(x),\\  \iota_{\lambda} {f}_a(x)&=\begin{cases} {f}^*_a\iota_{\lambda}(x) & \text{ if }\varphi_a(x)>0 \\ 0 & \text{ else.}\end{cases}
\end{split}
\end{align}
\end{lem}
In Theorem \ref{crossing} we give a formula for the crystal structure of Lemma \ref{staronpoln}.

\section{Wiring diagrams and Reineke crossings}\label{wiringandrein}

Following \cite{BFZ}, we introduce the notion of a wiring diagram which is a graphical presentation of the reduced word $\ii\in \W$.

\begin{defi}[wiring diagram]\label{defi:wire} 
Let $\ii=(i_1,i_2,\ldots i_N)\in \W$. The \emph{wiring diagram} $\mathcal{D}_{\ii}$ consists of a family of $n$ piecewise straight lines, called \emph{wires}, which can be viewed as graphs of $n$ continuous piecewise linear functions defined on the same interval. The wires have labels in the set $[n]$. Each vertex of $\mathcal{D}_{\ii}$ (i.e. an intersection of two wires) represents a letter $j$ in $\ii$. If the vertex corresponds to the letter $j\in [n-1]$, then $j-1$ is equal to the number of wires running below this intersection. We call $j$ the \emph{level of the vertex} $v$ and write
$$\text{level}(v)=j-1.$$

The word $\ii$ can be read off from $\mathcal{D}_{\ii}$ by reading the levels of the vertices from left to right. 
\end{defi}

\begin{ex}\label{run} Let $n=5$ and $\ii=(2,1,2,3,4,3,2,1,3,2)$. The corresponding wiring diagram $\mathcal{D}_{\ii}$ is depicted below.
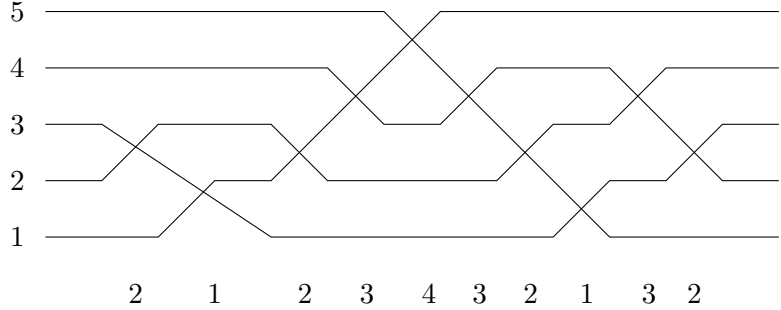
\begin{center}

\begin{tikzpicture}[scale=.75]

\node at (-.5,0) {$1$};
\node at (-.5,1) {$2$};
\node at (-.5,2) {$3$};
\node at (-.5,3) {$4$};
\node at (-.5,4) {$5$};
\node at(1.6,-1){$2$};
\node at(3,-1){$1$};
\node at(4.6,-1){$2$};
\node at(5.7,-1){$3$};
\node at(6.8,-1){$4$};
\node at(7.7,-1){$3$};
\node at(8.6,-1){$2$};
\node at(9.6,-1){$1$};
\node at(10.7,-1){$3$};
\node at(11.5,-1){$2$};

\draw (0,0) --(1,0) --  (2,0) -- (3,1) -- (4,1) -- (5,2) -- (6,3) --(7,4) -- (10,4) -- (13,4);
\draw (0,1) -- (1,1) -- (2,2) -- (4,2) -- (5,1) -- (8,1) -- (9,2) -- (10,2) --(11,3) -- (13,3);
\draw (0,2) --(1,2) -- (4,0) -- (9,0) -- (10,1) -- (11,1) -- (12,2) -- (13,2) ;
\draw (0,3) --  (5,3) -- (6,2) -- (7,2) -- (8,3) --(10,3) -- (11,2) -- (12,1) -- (13,1) ;
\draw (0,4) --  (6,4) -- (7,3) -- (8,2) -- (9,1) -- (10,0) -- (13,0) ;
\end{tikzpicture}
\end{center}

\end{ex}
The condition $\ii\in \W$ implies that two lines $p,q$ with $p\ne q$ in $\mathcal{D}_{\ii}$ intersect exactly once. 

Each vertex of the wiring diagram $\mathcal{D}_{\ii}$, $\ii\in \W$, corresponds to an inversion $(p,q)\in I(w_0)$, where $p$ and $q$ are the labels of the wires intersecting in that vertex. Thus the vertices of $\mathcal{D}_{\ii}$ are in bijection with the positive roots by \eqref{vertexroot}. The reflection order on  $I(w_0)$ and the induced total order on $\Phi^+$ can be read off of $\mathcal{D}_{\ii}$ by reading the vertices from left to right. We identify
\begin{equation}
\label{posrootsorder}
[N] \leftrightarrow I(w_0)=\left\{(p,q)\in [n]^2 \,\middle|\, p<q \right\}
\end{equation}
such that $k\in[N]$ corresponds to the $k$-th vertex $(p,q)\in I(w_0)$ in  $\mathcal{D}_{\ii}$ from left.

\begin{ex} We continue with Example \ref{run}. The reflection ordering
$$(2,3)<(1,3)<(1,2)<(1,4)<(1,5)<(4,5)<(2,5)<(3,5)<(2,4)$$
corresponding to $\ii$ is depicted in the wiring diagram $\mathcal{D}_{\ii}$ below.

\begin{center}

\begin{tikzpicture}[scale=.75]

\node at (-.5,0) {$1$};
\node at (-.5,1) {$2$};
\node at (-.5,2) {$3$};
\node at (-.5,3) {$4$};
\node at (-.5,4) {$5$};

\draw (0,0) --(1,0) -- node[xshift=1.05cm,yshift=-0.05cm,above]{$\scriptstyle{(1,3)}$} node[xshift=6.05cm,yshift=-0.3cm,above]{$\scriptstyle{(3,5)}$} (2,0) -- (3,1) -- (4,1) -- (5,2) -- (6,3) --(7,4) -- (10,4) -- (13,4);
\draw (0,1) -- (1,1) -- (2,2) -- (4,2) -- (5,1) -- (8,1) -- (9,2) -- (10,2) --(11,3) -- (13,3);
\draw (0,2) -- node[xshift=0.85cm,yshift=-1cm,above]{$\scriptstyle{(2,3)}$} node[xshift=3.05cm,yshift=-1.1cm,above]{$\scriptstyle{(1,2)}$} node[xshift=6.05cm,yshift=-1.1cm,above]{$\scriptstyle{(2,5)}$} (1,2) -- (4,0) -- (9,0) -- (10,1) -- (11,1) -- (12,2) -- (13,2) ;
\draw (0,3) -- node[xshift=2.28cm,yshift=-1.1cm,above]{$\scriptstyle{(1,4)}$} node[xshift=3.8cm,yshift=-1.1cm,above]{$\scriptstyle{(4,5)}$} node[xshift=6.1cm,yshift=-1.1cm,above]{$\scriptstyle{(2,4)}$} (5,3) -- (6,2) -- (7,2) -- (8,3) --(10,3) -- (11,2) -- (12,1) -- (13,1) ;
\draw (0,4) -- node[xshift=2.65cm,yshift=-1cm,above]{$\scriptstyle{(1,5)}$} (6,4) -- (7,3) -- (8,2) -- (9,1) -- (10,0) -- (13,0) ;
\end{tikzpicture}
\end{center}
\end{ex}

\begin{defi} Let $\ii\in \W$ and $\mathcal{D}_{\ii}$ be the corresponding wiring diagram. For $a\in [n-1]$ we denote by $\mathcal{D}_{\ii}(a)$ the oriented graph obtained from $\mathcal{D}_{\ii}$ by orienting its wires $p$ from left to right if $p\le a$, and from right to left if $p>a$. 
\end{defi}

\begin{ex} Let $a=3$ and $\mathcal{D}_{\ii}$ as in Example \ref{run}. The oriented graph $\mathcal{D}_{\ii}(3)$ looks as follows.

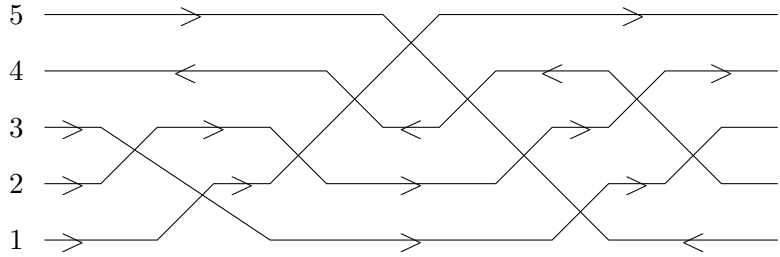
\begin{center}

\begin{tikzpicture}[scale=.75]

\node at (-.5,0) {$1$};
\node at (-.5,1) {$2$};
\node at (-.5,2) {$3$};
\node at (-.5,3) {$4$};
\node at (-.5,4) {$5$};

\draw (0,0) -- node[yshift=-0.033cm]{\textbf{>}} (1,0) -- (2,0) -- (3,1) -- node[yshift=-0.033cm]{\textbf{>}} (4,1) -- (5,2) -- (6,3) --(7,4) -- (10,4) -- node[xshift=-0.8cm,yshift=-0.033cm]{\textbf{>}} (13,4);
\draw (0,1) -- node[yshift=-0.033cm]{\textbf{>}} (1,1) -- (2,2) -- node[yshift=-0.033cm]{\textbf{>}} (4,2) -- (5,1) -- node[yshift=-0.033cm]{\textbf{>}} (8,1) -- (9,2) -- node[yshift=-0.033cm]{\textbf{>}} (10,2) --(11,3) -- node[yshift=-0.033cm]{\textbf{>}} (13,3);
\draw(0,2) -- node[yshift=-0.033cm]{\textbf{>}} (1,2) -- (4,0) -- node[yshift=-0.033cm]{\textbf{>}} (9,0) -- (10,1) -- node[yshift=-0.033cm]{\textbf{>}} (11,1) -- (12,2) -- (13,2) ;
\draw (0,3) -- node[yshift=-0.033cm]{\textbf{<}} (5,3) -- (6,2) -- node[yshift=-0.033cm]{\textbf{<}} (7,2) -- (8,3) -- node[yshift=-0.033cm]{\textbf{<}} (10,3) -- (11,2) -- (12,1) -- (13,1) ;
\draw (0,4) -- node[xshift=-0.3cm,yshift=-0.033cm]{\textbf{>}} (6,4) -- (7,3) -- (8,2) -- (9,1) -- (10,0) -- node[yshift=-0.033cm]{\textbf{<}} (13,0) ;
\end{tikzpicture}
\end{center}

\end{ex}

An oriented path in $\mathcal{D}_{\ii}(a)$ is a sequence $(v_1,\ldots,v_k)$ of vertices of $\mathcal{D}_{\ii}$ which are connected by oriented edges $v_1 \rightarrow v_1 \rightarrow \ldots \rightarrow v_{k}$ in $\mathcal{D}_{\ii}(a)$. 

\begin{defi}[Reineke crossings]\label{def:rigpath} For $a\in [n-1]$ an \emph{$a$-crossing} is an oriented path $\gamma=(v_1,\ldots,v_k)$ in $\mathcal{D}_{\ii}(a)$ which starts with the leftmost vertex of the wire $a$ and ends with the leftmost vertex of the wire $a+1$. Additionally $\gamma$ satisfies the following condition: Whenever $v_j, v_{j+1}, v_{j+2}$ lie on the same wire $p$ in $\mathcal{D}_{\ii}$ and the vertex $v_{j+1}$ lies on the intersection the wires $p$ and $q$, we have
\begin{align*}
p > q & \quad \text{if }q\le a \\ 
p<q & \quad \text{if }a+1 \le q. \end{align*}
In other words, the path $\gamma$ avoids the following two fragments.
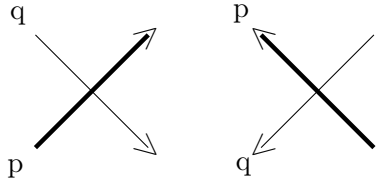
\begin{center}
\begin{tikzpicture}[scale=.75]

\draw[line width=0.6mm] (0,0)node[below left]{p} -- (2,2)node[thick]{\large{\rotatebox{40}{\textbf{>}}}};
\draw (2,0)node[thick]{\large{\rotatebox{320}{\textbf{>}}}}  -- (0,2)node[above left]{q};
\draw (4,0) node[below left]{q} node{\large{\rotatebox{220}{\textbf{>}}}} -- (6,2);
\draw[line width=0.6mm] (6,0) -- (4,2)node[above left]{p}node[thick]{\large{\rotatebox{140}{\textbf{>}}}};
\end{tikzpicture}
\end{center}
We denote the set of all $a$-Reineke crossings by $\Gamma_a$.
\end{defi}

\begin{rem} Reineke crossings appear as rigorous paths in \cite{GP}.
\end{rem}

\begin{ex}\label{rigpathex}
Let $n=5$. The vertices lying on the red path below form the $3-$Reineke crossing $\gamma=(v_{3,2},v_{3,1},v_{1,2},v_{2,5},v_{2,4},v_{4,5},v_{4,1})$.

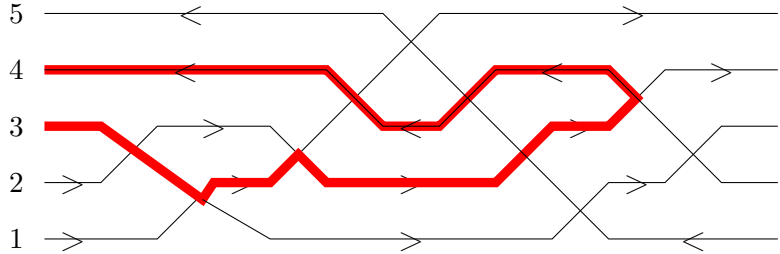
\begin{center}

\begin{tikzpicture}[scale=.75]

\node at (-.5,0) {$1$};
\node at (-.5,1) {$2$};
\node at (-.5,2) {$3$};
\node at (-.5,3) {$4$};
\node at (-.5,4) {$5$};

\draw (0,0) -- node[yshift=-0.033cm]{\textbf{>}} (1,0) -- (2,0) -- (3,1) -- node[yshift=-0.033cm]{\textbf{>}} (4,1) -- (5,2) -- (6,3) --(7,4) -- (10,4) -- node[xshift=-0.8cm,yshift=-0.033cm]{\textbf{>}} (13,4);
\draw (0,1) -- node[yshift=-0.033cm]{\textbf{>}} (1,1) -- (2,2) -- node[yshift=-0.033cm]{\textbf{>}} (4,2) -- (5,1) -- node[yshift=-0.033cm]{\textbf{>}} (8,1) -- (9,2) -- node[yshift=-0.033cm]{\textbf{>}} (10,2) --(11,3) -- node[yshift=-0.033cm]{\textbf{>}} (13,3);
\draw[line width=1.25mm, red] (0,2) -- (1,2)-- (2.8,0.7) -- (3,1) --  (4,1) -- (4.5,1.5) -- (5,1) -- (8,1) -- (9,2) -- (10,2) -- (10.5,2.5) -- (10,3) -- (8,3) -- (7,2) -- (6,2) -- (5,3) -- (0,3); 
\draw (2.8,0.7) -- (4,0) -- node[yshift=-0.033cm]{\textbf{>}} (9,0) -- (10,1) -- node[yshift=-0.033cm]{\textbf{>}} (11,1) -- (12,2) -- (13,2) ;
\draw (0,3) -- node[yshift=-0.033cm]{\textbf{<}} (5,3) -- (6,2) -- node[yshift=-0.033cm]{\textbf{<}} (7,2) -- (8,3) -- node[yshift=-0.033cm]{\textbf{<}} (10,3) -- (11,2) -- (12,1) -- (13,1) ;
\draw (0,4) -- node[xshift=-0.3cm,yshift=-0.033cm]{\textbf{<}} (6,4) -- (7,3) -- (8,2) -- (9,1) -- (10,0) -- node[yshift=-0.033cm]{\textbf{<}} (13,0) ;
\end{tikzpicture}

\end{center}
\end{ex}

In the remainder of this section we adopt the following convention: We label each vertex $v=v_{p,q}\in\gamma$ by the wires $p$ and $q$ that intersect in this edge where $p$ is the wire of the oriented edge whose source in $\gamma$ is $v_{p,q}$.

\begin{defi}\label{turn} Let $a\in [n-1]$ and $\gamma=(v_{p_1,q_1},v_{p_2,q_2},\ldots,v_{p_m,q_m})\in \Gamma_a$. We call the set of vertices $v_{p_s,q_s}$ such that $p_{s+1}=q_s$ the turning points $T_{\gamma}$ of $\gamma$. 
\end{defi}

\begin{ex} For $\gamma=(v_{3,2},v_{3,1},v_{1,2},v_{2,5},v_{2,4},v_{4,5},v_{4,1})$ as in Example \ref{rigpathex} we have $T_{\gamma}=\{v_{3,1},v_{1,2},v_{2,4} \}$.
\end{ex}

Using the identification \eqref{posrootsorder} we introduce
\begin{defi}\label{def:vectors} The maps $r:\Gamma_a \rightarrow \mathbb{Z}^N$ and $s:\Gamma_a \rightarrow \mathbb{Z}^N$ are given by		
\begin{align*}
\left(r({\gamma})\right)_{p,q} &:= \begin{cases} \sgn(q-p), &  \text{if $v_{p,q}\in T_{\gamma},$} \\
 0, & \text{else,}\end{cases} \\
\left(s({\gamma})\right)_{p,q} &:= \begin{cases} 1, & \text{if $v_{p,q}\in \gamma,$ $p\le a <q \text{ or }q\le a <p,$} \\
-1, & v_{p,q}\in \gamma\setminus T_{\gamma}, \ a<p,q \text{ or } p,q\le a,
\\ 0 & \text{else.}\end{cases}
\end{align*}
\end{defi}

\begin{ex} Let $\gamma=(v_{3,2},v_{3,1},v_{1,2},v_{2,5},v_{2,4},v_{4,5},v_{4,1})$ be as in Example \ref{rigpathex}. We have $$r({\gamma})=(0,-1,1,0,0,0,0,0,1,0), \quad s({\gamma})=(-1,0,0,1,0,-1,1,0,1,0).$$
\end{ex}

By \cite[Proposition 2.2]{GKS1} we have the following order relation $\preceq$ on $\Gamma_a$:
\begin{defi} Let $\gamma_1,\gamma_2\in \Gamma_a$. We say $\gamma_1 \preceq \gamma_2$ if all vertices of $\gamma_1$ lie in the region of $\mathcal{D}_{\ii}$ cut out by $\gamma_2$.
\end{defi}

\begin{ex} 
	Let $\gamma$ be as in Example \ref{rigpathex} and $\gamma'=(v_{3,2},v_{2,1},v_{1,4})$. In the picture below the region cut out by $\gamma$ is shaded grey while $\gamma'$ consists of all vertices lying on the red path. Thus $\gamma'\preceq \gamma$.
	
	\begin{center}
		
		\begin{tikzpicture}[scale=.75]
		
		\node at (-.5,0) {$1$};
		\node at (-.5,1) {$2$};
		\node at (-.5,2) {$3$};
		\node at (-.5,3) {$4$};
		\node at (-.5,4) {$5$};
		
		\draw[fill=lightgray!50] plot[smooth, samples=100, domain=1:e]  (0,2) -- (1,2)-- (2.8,0.7) -- (3,1) --  (4,1) -- (4.5,1.5) -- (5,1) -- (8,1) -- (9,2) -- (10,2) -- (10.5,2.5) -- (10,3) -- (8,3) -- (7,2) -- (6,2) -- (5,3) -- (0,3);
		\draw[line width=1.25mm, red] (0,2) -- (1,2)-- (1.55,1.55) -- (2.1,2)--(4.0,2)--(4.5,1.5)--(5.49,2.5)--(5.0,3.0)--(0,3);
		\draw (0,0) -- node[yshift=-0.033cm]{\textbf{>}} (1,0) -- (2,0) -- (3,1) -- node[yshift=-0.033cm]{\textbf{>}} (4,1) -- (5,2) -- (6,3) --(7,4) -- (10,4) -- node[xshift=-0.8cm,yshift=-0.033cm]{\textbf{>}} (13,4);
		\draw (0,1) -- node[yshift=-0.033cm]{\textbf{>}} (1,1) -- (2,2) -- node[yshift=-0.033cm]{\textbf{>}} (4,2) -- (5,1) -- node[yshift=-0.033cm]{\textbf{>}} (8,1) -- (9,2) -- node[yshift=-0.033cm]{\textbf{>}} (10,2) --(11,3) -- node[yshift=-0.033cm]{\textbf{>}} (13,3);
		\draw (2.8,0.7) -- (4,0) -- node[yshift=-0.033cm]{\textbf{>}} (9,0) -- (10,1) -- node[yshift=-0.033cm]{\textbf{>}} (11,1) -- (12,2) -- (13,2) ;
		\draw (0,3) -- node[yshift=-0.033cm]{\textbf{<}} (5,3) -- (6,2) -- node[yshift=-0.033cm]{\textbf{<}} (7,2) -- (8,3) -- node[yshift=-0.033cm]{\textbf{<}} (10,3) -- (11,2) -- (12,1) -- (13,1) ;
		\draw (0,4) -- node[xshift=-0.3cm,yshift=-0.033cm]{\textbf{<}} (6,4) -- (7,3) -- (8,2) -- (9,1) -- (10,0) -- node[yshift=-0.033cm]{\textbf{<}} (13,0) ;
		
		\end{tikzpicture}
		
	\end{center}
\end{ex}

\section{Dual Crossing Formula for string parametrizations}\label{sec:mainthm}
Let $\lambda \in P^+$ and $\ii\in \W$. In this section we state our main result which is a formula for the crystal structure on the integer points of the \lsp string polytope $\mathcal{S}^*_{\ii}(\lambda)^{\mathbb{R}}$ defined in \eqref{litstring}.

Recall the notion of the set of $a$-Reineke crossings $\Gamma_a$ from Definition \ref{def:rigpath} and their associated vectors from Definition \ref{def:vectors}. We denote by $\langle\cdot ,\cdot\rangle$ the standard scalar product on $\mathbb{Z}^N$. The crystal structure on $\mathcal{S}^*_{\ii}(\lambda)$ from Lemma \ref{staronpoln} is explicitly computed by 
\begin{thm}\label{crossing} For $\lambda\in P^+$, $a\in[n-1]$ and $x\in \mathcal{S}^*_{\ii}(\lambda)$ we have 
	\begin{align} \label{epsilon}
	\varepsilon_a(x)&=\max\left\{\left\langle x,r({\gamma})\right\rangle \,\middle\mid\, \gamma\in \Gamma_a\right\},\\\label{wts}
	\wt(x)&=\lambda  - \displaystyle\sum_{k\in [N]}x_k\alpha_{i_k},\\
	 \label{f}
	{f}_a (x)& = \begin{cases} x+s({\gamma^{x}}), & \text{if } \varphi_a(x)>0,\\
	0, & \text{else,} \end{cases} \\ \label{e}
	{e}_a (x) &= \begin{cases} x+s({\gamma_{x}}), & \text{if }\varepsilon_a(x)>0,\\
	0, & \text{else,} \end{cases}
	\end{align}
where $\gamma^{x}\in\Gamma_a$ is minimal with $\langle x,r({\gamma^{x}})\rangle=\varepsilon_a(x)$ and  $\gamma_{x}\in\Gamma_a$ is maximal with $\langle x,r({\gamma_{x}})\rangle=\varepsilon_a(x)$.
\end{thm}

Theorem \ref{crossing} is proved in Section \ref{mainproof}. A formula for the $*$-crystal structure on $\mathcal{S}_{\ii}$ given in Lemma \ref{eq:stringstarop} can directly deduced from Theorem \ref{crossing}:
\begin{thm}[Dual Crossing Formula]\label{stringinfty} For $a\in [n-1]$ and $x\in \mathcal{S}_{\ii}$ we have 
\begin{align*} 
\varepsilon_a^*(x)&=\max\left\{\left\langle x,r({\gamma})\right\rangle \,\middle\mid\, \gamma\in \Gamma_a\right\},\\
{f}_a^* (x) &= x+s({\gamma^{x}}),\\
{e}_a^* (x) &= \begin{cases} x+s({\gamma_{x}}), & \text{if }\varepsilon_a(x)>0, \\
0, & \text{else,} \end{cases}
\end{align*}
where $\gamma^{x}\in\Gamma_a$ is minimal with $\langle x,r({\gamma^{x}})\rangle=\varepsilon_a^*(x)$ and  $\gamma_{x}\in\Gamma_a$ is maximal with $\langle x,r({\gamma_{x}})\rangle=\varepsilon_a^*(x)$.
\end{thm}
\begin{proof} 	
	Since $\mathcal{S}_{\ii}=\cup_{\lambda\in P^+} \mathcal{S}^*_{\ii}(\lambda)$ we can find for each $x\in \mathcal{S}_{\ii}$ a $\lambda\in P^+$ such that $f_a^* x\in  \mathcal{S}^*_{\ii}(\lambda)=\{x \in \mathcal{S}_{\ii} \mid \varepsilon_{a}(x)\le \lambda_a \ \forall a \in [n-1]\}$. Thus the claim follows from Lemma \ref{staronpoln} and Theorem \ref{crossing}.
\end{proof}
\begin{rem}
The $*$-crystal structure on the string cone $\mathcal{S}_{\ii}$ is dual to the crystal structure on Lusztig data, which is governed by the Crossing Formula \ref{crossL} recalled below. By duality we understand the following: Maximum and minimum swap place as do the maps $r:\Gamma_a \rightarrow \Z^N$ and $s:\Gamma_a\rightarrow \Z^N$.

The $*$-crystal structure on Lusztig data $x\in\N^N$ is described by the $*$-Crossing Formula \cite[Theorem 2.20]{GKS1}, which is completely analogous to the Crossing Formula for Lusztig data. In \cite[Theorem 4.4]{GKS1} we show that $\mathcal{S}_{\ii}$ is polar to the set 
\begin{equation}
\label{slpolar}
\mathbf{R}^*=\left\{f_a^*x -x \,\middle|\, a\in[n-1],\, x\in\N^N \right\},
\end{equation}
i.e. the vectors $f_a^*x-x$ of the $*$-crystal structure on Lusztig data provide defining inequalities for $\mathcal{S}_{\ii}$. For the special case of reduced words adapted to quivers \eqref{slpolar} was obtained in \cite{Ze}.

Similarly, the set of Lusztig data $\N^N$ is polar to
$$
\left\{f_a x -x \,\middle|\, x\in\mathcal{S}_{\ii} \right\} = \left\{(\delta_{k, \ell})_{k\in[N]} \,\middle|\, \ell\in [N] \right\},
$$
i.e. the vectors $f_a x - x$ of the crystal structure \eqref{scrys} on $\mathcal{S}_{\ii}$ provide defining inequalities for the cone of Lusztig data $\N^N$.
\end{rem}
	
\section{Defining inequalities of Nakashima-Zelevinsky string polytopes}\label{sec:NZ}
Theorem \ref{crossing} provides a formula for the crystal structure on the \lsp string polytope $\mathcal{S}^*_{\ii}(\lambda)$. Switching the roles of $B(\infty)$ and $B(\infty)^*$ in the definition of $\mathcal{S}^*_{\ii}(\lambda)$ one arrives at
\begin{equation*}
\mathcal{S}_{\ii}(\lambda):=\{x\in \mathcal{S}_{\ii} \mid \varepsilon^*_a(x) \le \lambda_a \ \forall a \in [n-1]\}.
\end{equation*}
Building up on \cite{NZ}, $\mathcal{S}_{\ii}(\lambda)$ and its crystal structure is defined in \cite{N99}.

By Lemma \ref{eq:stringstarop} the set $\mathcal{S}_{\ii}(\lambda)$ consists of the integer points of the \emph{Nakashima-Zelevinsky string polytope} 
\begin{equation*}
\mathcal{S}_{\ii}(\lambda)^{\R}:=\{x\in \mathcal{S}_{\ii}^{\R} \mid \varepsilon^*_a(x) \le \lambda_a \ \forall a \in [n-1]\},
\end{equation*}
where $\varepsilon^*_a$ on $\mathcal{S}_{\ii}(\lambda)^{\R}$ is defined as in \eqref{epsiastarstr}. By \cite{FN} the convex polytope $\mathcal{S}_{\ii}(\lambda)^{\R}$ is rational. In this section we solve the problem of deriving defining inequalities for $\mathcal{S}_{\ii}(\lambda)^{\R}\subset\R^N.$

The Dual Crossing Formula (Theorem \ref{stringinfty}) immediately implies
\begin{thm}\label{inequpol} The set $\mathcal{S}_{\ii}(\lambda)^{\mathbb{R}}\subset \mathcal{S}_{\ii}^{\mathbb{R}}$ is explicitly described by
$$\mathcal{S}_{\ii}(\lambda)^{\mathbb{R}}=\left\{x \in \mathcal{S}_{\ii}^{\mathbb{R}} \,\middle\mid \,
 \left\langle x,r({\gamma}) \right\rangle \le \lambda_a \text{ for all }a\in [n-1] \text{ and for all } \gamma \in \Gamma_a \right\}.$$
\end{thm}

Using the explicit description of defining inequalities of $\mathcal{S}_{\ii}^{\mathbb{R}}$ obtained in \cite{GP} we obtain defining inequalities of $\mathcal{S}_{\ii}(\lambda)^{\mathbb{R}}\subset \mathbb{R}^N$. We recall the result of \cite{GP} for the convenience of the reader. 

Using the notation of Section \ref{wiringandrein} let $\mathcal{D}_{\ii}$ be the wiring diagram associated to $\ii \in \W$. For $a\in [n-1]$ let $\mathcal{D}_{\ii}(a)^{\vee}$ be the graph obtained from $\mathcal{D}_{\ii}(a)$ by reversing all arrows. For $a\in [n-1]$ an \emph{$a$-rigorous path} is an oriented path $\gamma=(v_1,\ldots,v_k)$ in $\mathcal{D}_{\ii}(a)^{\vee}$ which starts with the rightmost vertex of the wire $a$ and ends with the rightmost vertex of the wire $a+1$. Additionally $\gamma$ satisfies the following condition: Whenever $v_j, v_{j+1}, v_{j+2}$ lie on the same wire $p$ in $\mathcal{D}_{\ii}$ and the vertex $v_{j+1}$ lies on the intersection the wires $p$ and $q$, we have
\begin{align*} p > q & \quad \text{if }q\le a \\ 
p<q & \quad \text{if }a+1 \le q. \end{align*}
We denote the set of all $a$-rigorous paths by $\riga$.

For $\gamma\in \riga$ we define the set of turning points and the vector $r({\gamma})$ as in Definitions \ref{turn} and \ref{def:vectors}, respectively.

As a direct consequence of \cite[Corollary 5.8]{GP} and Theorem \ref{inequpol} we obtain
\begin{cor}\label{inequ2} The Nakashima-Zelevinsky string polytope $\mathcal{S}_{\ii}(\lambda)^{\mathbb{R}}$ is explicitly described by
$$\mathcal{S}_{\ii}(\lambda)^{\mathbb{R}} =\left\{ x\in \mathbb{R}^N \,\middle\mid\, \left\langle x, r({\gamma}) \right\rangle \ge 0, \ \left\langle x, r({\gamma'}) \right\rangle  \le \lambda_a \ \forall a\in [n-1], \, \gamma \in \riga,\,  \gamma'\in \Gamma_a\right\}.$$
\end{cor}

For the sake of completeness we recall the crystal structure on $\mathcal{S}_{\ii}(\lambda)$. For $k\in[N]$ we consider the function $\eta_k$ on $\mathcal{S}_{\ii}(\lambda)$ defined in \eqref{strstat1}.
Analogously to Lemma \ref{staronpoln} we have
\begin{lem}[\cite{N99}]\label{NZcrys}
	The following defines a crystal structure on $\mathcal{S}_{\ii}(\lambda)$ isomorphic to $B(\lambda)$.
	For $x\in \mathcal{S}_{\ii}(\lambda)$ and $a\in[n-1]$
\begin{align*}\varepsilon_a(x)&=\max\left\{\eta_k(x) \mid k\in[N],\, i_k=a \right\},\qquad 
\wt(x)=\lambda- \sum_{k\in [N]} x_k \alpha_{i_k},\\
{f}_a (x) &=\begin{cases} x+\left(\delta_{k,\ell_{x}}\right)_{k\in[N]}& \text{if }\varphi_a(x)>0,\\
0, & \text{else,} \end{cases} \\
{e}_a (x) & =\begin{cases} x-\left(\delta_{k, \ell^{x}}\right)_{k\in[N]}& \text{if }\varepsilon_a(x)>0 \\
0, & \text{else,} \end{cases}
\end{align*}
where $\ell^{x}\in[N]$ is minimal with $i_{\ell^x}=a$ and $\eta_{\ell^x}(x)=\varepsilon_a(x)$ and where
$\ell_{x}\in[N]$ is maximal with $i_{\ell_x}=a$ and $\eta_{\ell_x}(x)=\varepsilon_a(x)$.
\end{lem}

\section{The Crossing Formula on Lusztig data}\label{sec:Lus}
The main ingredient in the proof of Theorem \ref{crossing} is the Crossing Formula proved in \cite{GKS1}, which we recall in this section.

\subsection{Lusztig's parametrization of the canonical basis}\label{sec:crystal}
Lusztig \cite{Lu} associated to a reduced word $\ii=(i_1,i_2,\ldots,i_N)\in \W$ a PBW-type basis $B_{\ii}$ of $U_q^-$ as follows.
Let $\beta_{1}< \beta_2 <\ldots<\beta_N$ be the total ordering of $\Phi^+$ corresponding to $\ii$ via Remark \ref{convexorder}. We set 
$$F_{\ii, \beta_m}:=T_{i_1}T_{i_2}\cdots T_{i_{m-1}}F_{i_m},$$
where $T_i$ acts via the braid group action defined in \cite[Section 1.3]{Lu2}. The divided powers $x^{(m)}$ for $x\in U_q^{-}$ are defined in \eqref{dividedpower}. Then the PBW-type basis
$$\mathbf{B}_{\ii}:=\left\{F_{\ii,\beta_1}^{(x_{1})} F_{\ii,\beta_2}^{(x_{2})} \cdots F_{\ii,\beta_N}^{(x_{N})} \,\middle\mid\, (x_{1},x_{2},\ldots, x_{N})\in \mathbb{N}^N\right\}$$
is in natural bijection with the canonical basis $\mathbf{B}$ of $U_q^-$ (see \cite[Proposition 2.3, Theorem 3.2]{Lu}).

\begin{defi}We call $x=(x_{1},x_{2},\ldots, x_{N})\in \mathbb{N}^N$, the \emph{$\ii$-Lusztig datum} of the element $F_{\ii,\beta_1}^{(x_{1})} F_{\ii,\beta_2}^{(x_{2})} \cdots F_{\ii,\beta_N}^{(x_{N})}\in \mathbf{B}_{\ii}$.
\end{defi}

\subsection{Crystal structures on Lusztig's parametrizations}
Let $\ii$ and $\jj$ be two reduced words for $w_0$. A piecewise linear bijection $\RR_{\jj}^{\ii}:\mathbb{N}^{N}\rightarrow \mathbb{N}^{N}$ from the set of $\ii$-Lusztig data to the set of $\jj$-Lusztig data is defined in \cite[Section 2.1]{Lu} using the fact that any reduced word $\jj$ can be obtained from any other reduced word $\ii$ by applying a sequence of $2$- and $3$-moves given in Definition \ref{moves}.

Let $\ii\in\W$ with corresponding total ordering $\beta_{1}< \beta_2 <\ldots<\beta_N$ of $\Phi^+$ as in Remark \ref{convexorder}. The crystal structure on $\ii$-Lusztig data $\mathbb{N}^N$ obtained from $B(\infty)$ via the bijection 
\begin{equation}
\label{lusBiso}
(x_1, \dots, x_N) \mapsto b_{\ii}(x):=F_{\ii,\beta_1}^{(x_{1})} F_{\ii,\beta_2}^{(x_{2})} \cdots F_{\ii,\beta_N}^{(x_{N})}\in\mathbf{B}_{\ii}\simeq \mathbf{B}
\end{equation}
 is given as follows (see \cite{L93}, also \cite[Proposition 3.6]{BZ2}).
\begin{prop}\label{def:crystal} 
	Let $a\in[n-1]$ and $\jj\in\W$ with $j_1=a$. For an $\ii$-Lusztig datum $x\in \mathbb{N}^N$ and $y:=\Phi^{\ii}_{\jj}(x)$
	\begin{align*}
	\varepsilon_{a}(x)&=y_1,\qquad \wt(x)=-\sum_{k\in [N]}x_k \beta_k,\\
	f_{a}(x) &= \Phi^{\jj}_{\ii}\left(y+(1,0,0,\dots)\right),\\
	e_{a}(x) &= \begin{cases}
	\Phi^{\jj}_{\ii}\left(y-(1,0,0,\dots) \right), & \text{if $\varepsilon_a(x)>0,$}\\0,&\text{else.}
	\end{cases}
	\end{align*}
\end{prop}

The main result of \cite{GKS1} is the Crossing Formula for the crystal structure from Proposition \ref{def:crystal}. Using \eqref{blambdainbinfty} this leads for $\lambda \in P^+$ to a formula for the crystal structure on $\mathcal{L}_{\ii}(\lambda):=\{x\in \mathbb{N}^N \mid \varepsilon^*_a(x)\le \lambda_a \ \forall a \in [n-1]\}$ isomorphic to $B(\lambda)$: 
\begin{thm}[{\cite[Theorem 2.13, Proposition 2.20]{GKS1}}]\label{crossL}   For $\lambda\in P^+$, $x\in \mathcal{L}_{\ii}(\lambda)$ and $a\in[n-1]$ we have 
	\begin{align*} 
	\varepsilon_a(x)&=\max\left\{\left\langle x,s({\gamma})\right\rangle \,\middle\mid\, \gamma\in \Gamma_a\right\},\qquad 	\wt(x)=\lambda  - \displaystyle\sum_{k\in [N]}x_k \beta_k,\\ 
	{f}_a (x)& = \begin{cases} x+r({\gamma_{x}}), & \text{if } \varphi_a(x)>0,\\
	0, & \text{else,} \end{cases} \\ 
	{e}_a (x) &= \begin{cases} x+r({\gamma^{x}}), & \text{if }\varepsilon_a(x)>0,\\
	0, & \text{else,} \end{cases}
	\end{align*}
	where $\gamma^{x}\in \Gamma_a$ is minimal with $\langle x,s(\gamma^x)\rangle=\varepsilon_a(x)$ and $\gamma_{x}\in \Gamma_a$ is maximal with $\langle x,s(\gamma_x)\rangle=\varepsilon_a(x)$.
\end{thm}

\section{Proof of Theorem \ref{crossing}}\label{mainproof}
We fix $\ii=(i_1,\ldots,i_N) \in \W$ as well as $\lambda=\sum_{b\in [n]}\lambda_b\omega_b\in P^+$ and set
\begin{align*}
\lambda^*&:=\sum_{b \in [n]} \lambda_{n-b}\omega_b\in P^+,\\
\underline{\lambda}&:=(\lambda_{i_1},\lambda_{i_2},\ldots \lambda_{i_N})\in \mathbb{N}^N.
\end{align*}
\subsection{A bijection between string and Lusztig data}
Let  $(c_{i, j})$ be the Cartan matrix of $\sln$. For 
$x\in\Z^N$ we define
\begin{align*}
F_{\ii} (x)&:=\left(x_k+\displaystyle\sum_{k<\ell\leq N} c_{i_k,i_{\ell}}x_{\ell}\right)_{k\in [N]}\in\Z^N,\\
G_{\ii}^{\lambda}(x)&:=\underline{\lambda}-F_{\ii}(x)\in\Z^N.
\end{align*}

By \cite[MG03, Corollaire 3.5]{MG}, \cite[Lemma 6.3]{CMMG} 
(see also \cite[Lemma 6.4, Lemma 7.4, Proposition 8.2]{GKS2}) we have 
\begin{prop}\label{prop:SL} The map $G_{\ii}^{\lambda}$ restricts to a bijection
\begin{equation*}
G_{\ii}^{\lambda}:\mathcal{S}^*_{\ii}(\lambda) 
\xrightarrow{\sim} \mathcal{L}_{\ii}(\lambda^*).
\end{equation*}
Further, $G_{\ii}^{\lambda} \circ \Psi^{\ii}_{\jj}=\Phi^{\ii}_{\jj}\circ G_{\ii}^{\lambda}$ for any $\jj \in \W$.
\end{prop}

The bijection $G_{\ii}^{\lambda}$ between $\mathcal{S}^*_{\ii}(\lambda)$ and $\mathcal{L}_{\ii}(\lambda^*)$ intertwines the crystal structures given in Lemma \ref{staronpoln} and Proposition \ref{def:crystal} as follows.
\begin{lem}\label{crysint} For 
	$a\in[n-1]$ we have on $\mathcal{S}^*_{\ii}(\lambda)$
	\begin{align}\label{epscomp}
	\varepsilon_a &= \varphi_a \circ G_{\ii}^\lambda,\\\label{ecomp}
	G_{\ii}^{\lambda} \circ e_a &= f_a \circ G_{\ii}^{\lambda},\\\label{wtcomp}
	\wt&=-\wt\circ G_{\ii}^{\lambda}.
	\end{align}
\end{lem}

\begin{proof} Clearly, \eqref{epscomp} and \eqref{ecomp} hold for $i_1=a$ and thus by Proposition \ref{prop:SL} for arbitrary $\ii\in\W$.
	
By \eqref{ecomp} and the crystal axiom (C3) in Definition \ref{abstrcryst} it is enough to show \eqref{wtcomp} for the highest weight element $x_\lambda$ of $\mathcal{S}^*_{\ii}(\lambda)$. By \eqref{epscomp} we have for $a'\in[n-1]$
$$
\varphi_{a'} \circ G_{\ii}^\lambda (x_{\lambda}) = 	\varepsilon_{a'} (x_\lambda ) = 0,
$$
i.e. $G_{\ii}^\lambda (x_{\lambda})$ is the lowest weight element of $\mathcal{L}_{\ii}(\lambda^*)$. Thus
$$
\wt(x_\lambda) = \lambda = -\wt \circ G_{\ii}^\lambda (x_{\lambda}).
$$
\end{proof}

\subsection{Reineke crossings and the bijection $G_{\ii}^{\lambda}$}
For $a\in [n-1]$ we attach in Definition \ref{def:vectors} to $\gamma\in \Gamma_a$ the vectors $s({\gamma})$, $r({\gamma})\in \mathbb{Z}^N$. In \cite[Theorem 3.11]{G18} it is shown that the map $F_{\ii}$ relates $s({\gamma})$ and $r({\gamma})\in \mathbb{Z}^N$ as follows:
\begin{prop}[\cite{G18}]\label{Fs}
	For $a\in [n-1]$ we have 
	$r=F_{\ii}\circ s$ on $\Gamma_a$.
\end{prop}

In this section we use Proposition \ref{Fs} to show
\begin{prop}\label{wt} For $x\in \mathcal{S}^*_{\ii}(\lambda)$, $a\in [n-1]$ and $\gamma\in \Gamma_a$ we have
$$\left\langle G_{\ii}^{\lambda}(x),s({\gamma})\right\rangle-\left\langle x,r({\gamma})\right\rangle=\wt(x)(h_a).$$
\end{prop}
For this we define for $a\in [n-1]$ the function
\begin{align*}
\ell_a: \mathbb{Z}^N & \rightarrow \mathbb{Z} \\ 
x=(x_k)_{k\in [N]} & \mapsto \displaystyle\sum_{k: \ i_{k}=a}x_{k}.
\end{align*}
To prove Proposition \ref{wt} we use
\begin{lem}\label{genius} For $a,b \in [n-1]$ and $\gamma\in \Gamma_a$ we have $\ell_b(s({\gamma}))=\delta_{a,b}$.
\end{lem}
\begin{proof}[Proof of Proposition \ref{wt}]
From Proposition \ref{Fs} we obtain
\begin{align}\notag
\left\langle G_{\ii}^{\lambda}(x),s({\gamma})\right\rangle-\left\langle x,r({\gamma})\right\rangle&=
\left\langle\underline{\lambda}, s({\gamma})\right\rangle-\left\langle F_{\ii}(x),s({\gamma})\right\rangle-\left\langle x,r({\gamma})\right\rangle\\&=
\left\langle \underline{\lambda}, s({\gamma})\right\rangle -\left\langle F_{\ii}(x),s({\gamma})\right\rangle -\left\langle x,F_{\ii}(s({\gamma}))\right\rangle.
\label{1}
\end{align}
By Lemma \ref{genius} we have
\begin{equation}\label{2}
\left\langle\underline{\lambda}, s({\gamma})\right\rangle = \displaystyle\sum_{k \in [N]}  \lambda_{i_k} (s({\gamma}))_{k}
= \displaystyle\sum_{b \in [n-1]} \lambda_{b} \ell_b (s({\gamma}))=\lambda_a.
\end{equation}
Furthermore, since $c_{b,b}=2$,
\begin{align*} & \left\langle F_{\ii}(x),s({\gamma})\right\rangle+\left\langle x,F_{\ii}(s({\gamma}))\right\rangle 
=\displaystyle\sum_{k\in [N]} (F_{\ii}(x))_k(s({\gamma}))_k+\displaystyle\sum_{k\in [N]}x_k\left(F_{\ii}(s({\gamma}))\right)_{k} \\
&=\displaystyle\sum_{k\in [N]} \left(x_k+ \displaystyle\sum_{\ell > k}c_{i_k,i_{\ell}}x_{\ell}\right)(s({\gamma}))_{k}+\displaystyle\sum_{k\in [N]}x_k\left((s({\gamma}))_k  + \displaystyle\sum_{\ell>k}c_{i_k,i_{\ell}}(s({\gamma}))_{\ell}\right) \\
&=\displaystyle\sum_{k,\ell \in [N]} c_{i_k,i_{\ell}}x_k(s({\gamma}))_{\ell}=\displaystyle\sum_{i,j\in [n-1]} c_{i,j}\ell_i(x)\ell_j(s({\gamma})).
\end{align*}
Thus, by Lemma \ref{genius}
\begin{equation}\label{3}
\left\langle F_{\ii}(x),s({\gamma})\right\rangle+\left\langle x,F_{\ii}(s({\gamma}))\right\rangle=\displaystyle\sum_{i \in [n-1]} c_{a,i}\ell_i(x).
\end{equation}
Combining \eqref{1}, \eqref{2} and \eqref{3} yields
$$\left\langle G_{\ii}^{\lambda}(x),s({\gamma})\right\rangle-\left\langle x,r({\gamma})\right\rangle=\lambda_a-\displaystyle\sum_{i\in[n-1]}\sum_{k: \ i_k=i}c_{a,i}x_k=\wt(x)(h_a).$$
\end{proof}

It remains to prove Lemma \ref{genius}. Recall the notion of the level of a vertex $v$ of $\mathcal{D}_{\ii}$ from Definition \ref{defi:wire}. For each vertex $v$ of $\gamma$, we define
$$\text{level}_{\gamma}^-(v)=\begin{cases}\text{level}(v)+1& \text{ the oriented edge of }\mathcal{D}_{\ii}(a) \text{ with target }v \text{ that }\gamma \\ & \text{ follows}\text{ is headed downwards,} \\
\text{level}(v) & \text{ the oriented edge of }\mathcal{D}_{\ii}(a) \text{ with target }v \text{ that }\gamma  \\ & \text{ follows}\text{ is headed upwards,} \end{cases}$$
and
$$\text{level}_{\gamma}^+(v)=\begin{cases}\text{level}(v)& \text{ the oriented edge of }\mathcal{D}_{\ii}(a) \text{ with source }v \text{ that }\gamma \\ & \text{ follows}\text{ is headed downwards,} \\
\text{level}(v)+1 & \text{ the oriented edge of }\mathcal{D}_{\ii}(a) \text{ with source }v \text{ that }\gamma  \\ & \text{ follows}\text{ is headed upwards.} \end{cases}$$
Here we understand ''headed upwards'' and ''headed downwards'' with respect to a small neighborhood around the vertex $v$.

We give an example for this notion.
\begin{ex}
Let $n=5$. And $\gamma=(v_{3,2},v_{3,1},v_{1,2},v_{2,5},v_{2,4},v_{4,5},v_{4,1})$ the $3-$Reineke crossing from Example \ref{rigpathex} colored red below. We have \begin{align*} \text{level}_{\gamma}^-(v_{3,2})&=3, \ \text{level}_{\gamma}^+(v_{3,2})=2, \ \text{level}_{\gamma}^-(v_{3,1})=2, \ \text{level}_{\gamma}^+(v_{3,1})=2, \\ 
 \text{level}_{\gamma}^-(v_{1,2})& =2, \ \text{level}_{\gamma}^+(v_{1,2})=2, , \ \text{level}_{\gamma}^-(v_{2,5})=2, \ \text{level}_{\gamma}^+(v_{2,5})=3, \\ \text{level}_{\gamma}^-(v_{2,4})& =3, \ \text{level}_{\gamma}^+(v_{2,4})=4, \ \text{level}_{\gamma}^-(v_{4,5})=4, \ \text{level}_{\gamma}^+(v_{4,5})=3, \\ \text{level}_{\gamma}^-(v_{4,1})&=3, \ \text{level}_{\gamma}^+(v_{4,1})=4.
 \end{align*}
 
\begin{center}

\begin{tikzpicture}[scale=.75]

\node at (-.5,0) {$1$};
\node at (-.5,1) {$2$};
\node at (-.5,2) {$3$};
\node at (-.5,3) {$4$};
\node at (-.5,4) {$5$};

\draw (0,0) -- node[yshift=-0.033cm]{\textbf{>}} (1,0) -- (2,0) -- (3,1) -- node[yshift=-0.033cm]{\textbf{>}} (4,1) -- (5,2) -- (6,3) --(7,4) -- (10,4) -- node[xshift=-0.8cm,yshift=-0.033cm]{\textbf{>}} (13,4);
\draw (0,1) -- node[yshift=-0.033cm]{\textbf{>}} (1,1) -- (2,2) -- node[yshift=-0.033cm]{\textbf{>}} (4,2) -- (5,1) -- node[yshift=-0.033cm]{\textbf{>}} (8,1) -- (9,2) -- node[yshift=-0.033cm]{\textbf{>}} (10,2) --(11,3) -- node[yshift=-0.033cm]{\textbf{>}} (13,3);
\draw[line width=1.25mm, red] (0,2) -- (1,2)-- (2.8,0.7) -- (3,1) --  (4,1) -- (4.5,1.5) -- (5,1) -- (8,1) -- (9,2) -- (10,2) -- (10.5,2.5) -- (10,3) -- (8,3) -- (7,2) -- (6,2) -- (5,3) -- (0,3); 
\draw (2.8,0.7) -- (4,0) -- node[yshift=-0.033cm]{\textbf{>}} (9,0) -- (10,1) -- node[yshift=-0.033cm]{\textbf{>}} (11,1) -- (12,2) -- (13,2) ;
\draw (0,3) -- node[yshift=-0.033cm]{\textbf{<}} (5,3) -- (6,2) -- node[yshift=-0.033cm]{\textbf{<}} (7,2) -- (8,3) -- node[yshift=-0.033cm]{\textbf{<}} (10,3) -- (11,2) -- (12,1) -- (13,1) ;
\draw (0,4) -- node[xshift=-0.3cm,yshift=-0.033cm]{\textbf{<}} (6,4) -- (7,3) -- (8,2) -- (9,1) -- (10,0) -- node[yshift=-0.033cm]{\textbf{<}} (13,0) ;
\end{tikzpicture}

\end{center}
\end{ex}
Note that, by definition, for $\gamma=(v_1,v_2,\ldots,v_m)\in\Gamma_a$, we have $\text{level}_\gamma^-(v_1)=a$, $\text{level}_\gamma^+(v_{\ell})=\text{level}_\gamma^-(v_{\ell +1})$ and $\text{level}_\gamma^+(v_m)=a+1$. Thus, Lemma \ref{genius} is now a direct consequence of
\begin{lem}\label{pm} For $1 \le \ell \le m$ we have $\text{level}_\gamma^+(v_{\ell})-\text{level}_\gamma^-(v_{\ell})=(s({\gamma}))_{\ell}.$
\end{lem}
\begin{proof}Assume that the vertex $v_{\ell}=v_{p,q}$ of $\gamma$ lies at the intersection of wires $p$ and $q$, where $p$ is the oriented wire with source $v_{\ell}$.

We assume first $p\le a$, hence the wire $p$ is oriented from left to right in $\mathcal{D}_{\ii}(a)$. We proceed by a case by case analysis.

$\underline{q < p \le a}:$ Locally around $v_{\ell}$ there are two possibilities for $\gamma$:
\begin{center}
\begin{tikzpicture}[scale=.75]

\draw (0,0)node[below left]{{\color{black}q}} -- (2,2);
\draw[line width=0.6mm] (0,2) -- (1,1);
\draw (0,2) -- (1,1);
\draw (1,1) node[right]{{\color{black}$v_{\ell}$}} -- (2,2)node[thick]{\large{\rotatebox{40}{\textbf{>}}}};
\draw (2,0)node[thick]{\large{\rotatebox{320}{\textbf{>}}}}  -- (0,2)node[above left]{p};
\draw[line width=0.6mm] (1,1) -- (2,0);
\draw (1,1) -- (2,0);

\draw (4,0) node[below left]{q} -- (6,2)node[thick]{\large{\rotatebox{40}{\textbf{>}}}};
\draw (6,0) node[thick]{\large{\rotatebox{320}{\textbf{>}}}} -- (5,1);
\draw[line width=0.6mm] (5,1) -- (4,2);
\draw[line width=0.6mm] (5,1) -- (6,2);
\draw (5,1) node[right]{$v_{\ell}$};
\draw (4,2) node[above left]{p};

\end{tikzpicture}
\end{center}

In the left case, we have $(s({\gamma}))_{\ell}=-1$, $\text{level}_\gamma^-(v_{\ell})=\text{level}(v_{\ell})+1$ and $\text{level}_\gamma^+(v_{\ell})=\text{level}(v_{\ell})$. In the right case, we have $(s({\gamma}))_{\ell}=0$, $\text{level}_\gamma^-(v_{\ell})=\text{level}(v_{\ell})+1$ and $\text{level}_\gamma^+(v_{\ell})=\text{level}(v_{\ell})+1$.

$\underline{p < q \le a}:$ Locally around $v_{\ell}$ there are two possibilities for $\gamma$:
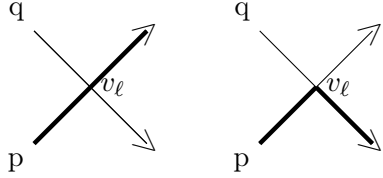
\begin{center}
\begin{tikzpicture}[scale=.75]

\draw (0,0)node[below left]{{\color{black}p}} -- (2,2);
\draw[line width=0.6mm] (0,0) -- (1,1);
\draw (0,2) -- (1,1);
\draw (1,1) node[right]{{\color{black}$v_{\ell}$}} -- (2,2)node[thick]{\large{\rotatebox{40}{\textbf{>}}}};
\draw (2,0)node[thick]{\large{\rotatebox{320}{\textbf{>}}}}  -- (0,2)node[above left]{q};
\draw[line width=0.6mm] (1,1) -- (2,2);
\draw (1,1) -- (2,0);

\draw (4,0) node[below left]{p} -- (6,2)node[thick]{\large{\rotatebox{40}{\textbf{>}}}};
\draw (6,0) node[thick]{\large{\rotatebox{320}{\textbf{>}}}} -- (5,1);
\draw[line width=0.6mm] (5,1) -- (4,0);
\draw[line width=0.6mm] (5,1) -- (6,0);
\draw (4,2) -- (5,1);
\draw (5,1) node[right]{$v_{\ell}$};
\draw (4,2) node[above left]{q};

\end{tikzpicture}
\end{center}

The left case cannot appear since $\gamma$ is an $a$-Reineke crossing. In the right case, we have $(s({\gamma}))_{\ell}=0$, $\text{level}_\gamma^-(v_{\ell})=\text{level}(v_{\ell})$ and $\text{level}_\gamma^+(v_{\ell})=\text{level}(v_{\ell})$.

$\underline{p \le a<q}:$ Locally around $v_{\ell}$ there are two possibilities for $\gamma$:

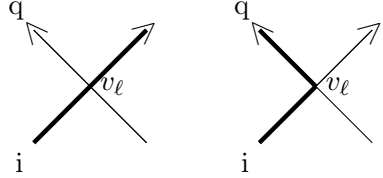
\begin{center}
\begin{tikzpicture}[scale=.75]

\draw (0,0)node[below left]{{\color{black}i}} -- (2,2);
\draw (0,2) -- (1,1);
\draw (1,1) node[right]{{\color{black}$v_{\ell}$}} -- (2,2)node[thick]{\large{\rotatebox{40}{\textbf{>}}}};
\draw (2,0)  -- (0,2)node[above left]{q}node[thick]{\large{\rotatebox{130}{\textbf{>}}}};
\draw (1,1) -- (2,0);
\draw[line width=0.6mm] (0,0) -- (2,2);

\draw (4,0) node[below left]{i} -- (6,2)node[thick]{\large{\rotatebox{40}{\textbf{>}}}};
\draw (6,0)  -- (5,1);
\draw[line width=0.6mm] (4,0) -- (5,1);
\draw[line width=0.6mm] (5,1) -- (4,2);
\draw (4,2) node[above left]{q}node[thick]{\large{\rotatebox{130}{\textbf{>}}}};
\draw(5,1) -- (6,2);
\draw (5,1) node[right]{$v_{\ell}$};

\end{tikzpicture}
\end{center}

In the both cases, we have $(s({\gamma}))_{\ell}=1$, $\text{level}_\gamma^-(v_{\ell})=\text{level}(v_{\ell})$ and $\text{level}_\gamma^+(v_{\ell})=\text{level}(v_{\ell})+1$. 

The argument for the assumption $a+1 \le p$ is symmetrical.
\end{proof}

\subsection{Proof of the Dual Crossing Formula}
\begin{proof}[Proof of Theorem \ref{crossing}]
Equation \eqref{wts} was established in Lemma \ref{staronpoln}.

We prove \eqref{epsilon}. By Lemma \ref{crysint} and the crystal axiom (C1) in Definition \ref{abstrcryst}
\begin{equation}\label{p2}
\varepsilon_a(x)=\varphi_a(G_{\ii}^{\lambda}(x))=\wt(G_{\ii}^{\lambda}(x))(h_a)+\varepsilon_a(G_{\ii}^{\lambda}(x)).
\end{equation}
By Proposition \ref{prop:SL} we have $G_{\ii}^{\lambda}(x)\in \mathcal{L}_{\ii}(\lambda^*)$. Using Theorem \ref{crossL} to compute the value of $\varepsilon_a$ on this Lusztig-datum we obtain
\begin{align} \notag
\varepsilon_a(G_{\ii}^{\lambda}(x))&=\max\{\left\langle G_{\ii}^{\lambda}(x),s({\gamma}) \right\rangle \mid \gamma\in \Gamma_a \} \\ \label{p1}
&= \max\{\left\langle x,r({\gamma}) \right\rangle \mid \gamma\in \Gamma_a\}+\wt(x)(h_a),
\end{align}
where \eqref{p1} follows from Proposition \ref{wt}. By Lemma \ref{crysint}
\begin{equation}\label{p3}
\wt(x)(h_a)=-\wt(G_{\ii}^{\lambda}(x))(h_a).
\end{equation}
Plugging \eqref{p1} and \eqref{p3} into \eqref{p2} yields \eqref{epsilon}.

We next prove \eqref{f}. If $\varphi_a(x)=0$ the claim follows from Lemma \ref{crysint}.

Assume now that $\varphi_a(x)>0$. By Lemma \ref{crysint} we have
\begin{equation}\label{f1}
{f}_ax={f}_a (G_{\ii}^{\lambda})^{-1}\circ G_{\ii}^{\lambda}(x)=(G_{\ii}^{\lambda})^{-1}\left({e}_aG_{\ii}^{\lambda}(x)\right).
\end{equation}
By Proposition \ref{prop:SL} we have that $G_{\ii}^{\lambda}(x)\in \mathcal{L}_{\ii}(\lambda^*)$ and by Lemma \ref{crysint} that $\varepsilon_a(x)>0$. 
Thus by Theorem \ref{crossL}
\begin{equation}\label{f3}
{e}_aG_{\ii}^{\lambda}(x) = G_{\ii}^{\lambda}(x) + r(\gamma^x),
\end{equation}
where $\gamma^{x}\in \Gamma_a$ is minimal with $\langle G_{\ii}^{\lambda}(x),s(\gamma^x)\rangle= \max\{\langle G_{\ii}^{\lambda}x,s({\gamma}) \rangle \mid \gamma\in \Gamma_a\}$. By Proposition \ref{wt} $\left\langle G_{\ii}^{\lambda}(x),s({\gamma}) \right\rangle-\left\langle x,r({\gamma})\right\rangle=\wt(x)(h_a)$ is independent of $\gamma\in \Gamma_a$. Thus, $\gamma^{x}\in\Gamma_a$ is minimal with 
$$
\left\langle x, r(\gamma^x)\right\rangle= \max\left\{\left\langle x,r({\gamma}) \right\rangle \mid \gamma\in \Gamma_a\right\}=\varepsilon_a(x),
$$
where we used \eqref{epsilon} in the last equality.
Furthermore, by \eqref{f1} and \eqref{f3}
\begin{equation*}
{f}_ax=(G_{\ii}^{\lambda})^{-1}\left(G_{\ii}^{\lambda}(x)+r(\gamma^x)\right)=x+F_{\ii}^{-1}(r(\gamma^x))
\end{equation*}
and \eqref{f} follows from Proposition \ref{Fs}.

The proof of \eqref{e} works analogously to the proof of \eqref{f}.
\end{proof}

\section{Kashiwara $*$-involution on String data}\label{sec:kashstar}
In this section we denote by $\mathcal{S}_{\ii}$ and $\mathcal{S}_{\ii}^*$ the set of $\ii$-string data equipped with the crystal structure inherited from $B(\infty)$ and $B(\infty)^*$, respectively, via the bijection $\str^*_{\ii}$ (see \eqref{scrys} and \eqref{epsiastarstr}). We denote by $\mathcal{L}_{\ii}=\N^N$ and $\mathcal{L}_{\ii}^*=\N^N$ the set of $\ii$-Lusztig data with the crystal structure inherited from $B(\infty)$ and $B(\infty)^*$, respectively,  via the bijection $b_{\ii}$ defined in \eqref{lusBiso}. We write $\mathcal{L}_{\ii}^{\mathbb{R}}:=\mathbb{R}_{\geq 0}^N$. Using $\varepsilon_a$ from the crystal $\mathcal{L}_{\ii}$ and $\varepsilon_a^*$ from $\mathcal{L}_{\ii}^*$ we define the polytopes
\begin{align*}
\mathcal{L}_{\ii}(\lambda)^{\mathbb{R}}&:=\{x\in \mathcal{L}_{\ii}^{\mathbb{R}} \mid \varepsilon^*_a(x)\le \lambda_a \ \forall a \in [n-1]\},\\
\mathcal{L}_{\ii}(\lambda)^{\mathbb{R}}&:=\{x\in \mathcal{L}_{\ii}^{\mathbb{R}} \mid \varepsilon_a(x)\le \lambda_a \ \forall a \in [n-1]\}.
\end{align*}
The integral points of $\mathcal{L}_{\ii}(\lambda)^{\mathbb{R}}$ and $\mathcal{L}_{\ii}^*(\lambda)^{\mathbb{R}}$ are $\mathcal{L}_{\ii}(\lambda)$ and $\mathcal{L}_{\ii}^*(\lambda)$ respectively.  

For a reduced word $\ii=(i_1, \dots, i_N)\in\W$ we define 
\begin{align*}
\ii^*&:=(n-i_1, \dots, n-i_N)\in\W,\\
\ii^{\text{op}}:&=(i_N, \dots, i_1)\in\W.
\end{align*}

For $\ii, \jj \in\W$ the Kashiwara $*$-involution  $* : B(\infty) \rightarrow B(\infty)^*$ introduced in Section \ref{sec:crysrep} on string data is given by the isomorphism of crystals
\begin{equation}\label{kashstr}
\str_{\jj}^*\circ\str_{\ii}^{-1}:\mathcal{S}_{\ii}^* \xrightarrow{\sim} \mathcal{S}_{\jj}.
\end{equation}
In general the map \eqref{kashstr} is piecewise linear. We show that \eqref{kashstr} is linear for $\ii=\ii_0:=(1,2,1,3,2,1, \dots, n-1,n-2,\dots, 1)$ and $\jj=\ii_0^{*}$.

Using the Crossing Formula \cite[Theorem 2.13]{GKS1}
 we compute $\str_{\ii_0}\circ b_{\ii_0}$: If $(i_{\ell},i_{\ell+1},\ldots,i_{\ell+m})$ is a maximal subword of $\ii_0$ of the form $(k,k-1,\ldots,1)$ we have for $j \in \{0,1,\ldots, m\}$ 
$$(\str_{\ii_0}\circ b_{\ii_0} (x))_{\ell+j}=x_{\ell}+x_{\ell+1}+\ldots x_{\ell+m-j}.$$
From the $*$-Crossing Formula \cite[Theorem 2.20]{GKS1} 
 we compute
$$
\str_{\ii_0^*}^*\circ b_{\ii_0^{\text{op}}} (x_1, \dots, x_N)=\str_{\ii_0}\circ b_{\ii_0} (x_N, \dots, x_1).
$$
Since $\ii_0$ and $\ii_0^{\text{op}}$ are related by a sequence of $2$-moves the isomorphism of crystals
$
\Phi^{\ii_0}_{\ii_0^{\text{op}}} 
$
sending $\ii_0$-Lusztig data to $\ii_0^{\text{op}}$-Lusztig data is linear. We thus obtain the linear isomorphism of crystals
$$
*=\str_{\ii_0^{*}}^*\circ\str_{\ii_0}^{-1}=\str_{\ii_0^{*}}^*\circ b_{\ii_0^{\text{op}}} \circ \Phi^{\ii_0}_{\ii_0^{\text{op}}}  \circ b_{\ii_0}^{-1}\circ \str_{\ii_0}^{-1} : \mathcal{S}_{\ii_0}^* \xrightarrow{\sim} \mathcal{S}_{\ii_0^*}.
$$
Since $*=\str_{\ii_0^{*}}^*\circ\str_{\ii_0}^{-1}:\mathcal{S}_{\ii_0} \xrightarrow{\sim} \mathcal{S}_{\ii_0^*}^*$ is an isomorphism of crystals as well, we obtain for $\lambda\in P^+$ the linear isomorphism of crystals
$$
*=\str_{\ii_0^{*}}^*\circ\str_{\ii_0}^{-1} :  \mathcal{S}_{\ii_0}^*(\lambda) \xrightarrow{\sim} \mathcal{S}_{\ii_0^*}(\lambda)
$$
and the unimodular isomorphismus of polytopes
$$
*=\str_{\ii_0^{*}}^*\circ\str_{\ii_0}^{-1} :  \mathcal{S}_{\ii_0}^*(\lambda)^{\mathbb{R}} \xrightarrow{\sim} \mathcal{S}_{\ii_0^*}(\lambda)^{\mathbb{R}}
$$

For $\ii, \jj\in\W$ arbitrary we obtain the piecewise linear isomorphisms
\begin{align*}
\Psi^{\ii_0^*}_{\jj} \circ \str_{\ii_0^{*}}^*\circ\str_{\ii_0}^{-1} \circ \Psi^{\ii}_{\ii_0}&: \mathcal{S}_{\ii}^* \xrightarrow{\sim} \mathcal{S}_{\jj},\\
\Psi^{\ii_0^*}_{\jj} \circ \str_{\ii_0^{*}}^*\circ\str_{\ii_0}^{-1} \circ \Psi^{\ii}_{\ii_0}&: \mathcal{S}_{\ii}^*(\lambda) \xrightarrow{\sim} \mathcal{S}_{\jj}(\lambda)
\end{align*}
and the piecewise linear volume preserving bijections
\begin{align*}
\Psi^{\ii_0^*}_{\jj} \circ \str_{\ii_0^{*}}^*\circ\str_{\ii_0}^{-1} \circ \Psi^{\ii}_{\ii_0}&: \mathcal{S}_{\ii}^{\mathbb{R}} \xrightarrow{\sim} \mathcal{S}_{\jj}^{\mathbb{R}},\\
\Psi^{\ii_0^*}_{\jj} \circ \str_{\ii_0^{*}}^*\circ\str_{\ii_0}^{-1} \circ \Psi^{\ii}_{\ii_0}&: \mathcal{S}_{\ii}^*(\lambda)^{\mathbb{R}} \xrightarrow{\sim} \mathcal{S}_{\jj}(\lambda)^{\mathbb{R}}.
\end{align*}

By \cite[Proposition 3.3 (iii)]{BZ2} the $*$-involution is given on Lusztig data by the linear map
\begin{align*}
*:\mathcal{L}_{\ii}&\xrightarrow{\sim} \mathcal{L}_{\ii^{*, \text{op}}}^*,\\
x=(x_1, \dots, x_N) &\mapsto x^{\text{op}}=(x_N, \dots, x_1).
\end{align*}
For $\lambda\in P^+$ we thus have the following commutative diagrams of isomorphisms of crystals which are linear for $\ii=\ii_0$.
\[ \begin{tikzcd}[row sep=normal, column sep=huge]
\mathcal{L}_{\ii}  \arrow{r}{\str_{\ii}\circ b_{\ii}} \arrow{d}{*} & \mathcal{S}_{\ii}^* \arrow{d}{*} \\%
\mathcal{L}^*_{\ii^{*\text{op}}} \arrow{r}{\str_{\ii^{*\text{op}}}\circ b_{\ii^{*\text{op}}}}& \mathcal{S}_{\ii^{*\text{op}}}
\end{tikzcd} 
\qquad\qquad 
\begin{tikzcd}[row sep=normal, column sep=huge]
\mathcal{L}_{\ii}(\lambda)  \arrow{r}{\str_{\ii}\circ b_{\ii}} \arrow{d}{*} & \mathcal{S}_{\ii}^*(\lambda) \arrow{d}{*} \\%
\mathcal{L}^*_{\ii^{*\text{op}}}(\lambda) \arrow{r}{\str_{\ii^{*\text{op}}}\circ b_{\ii^{*\text{op}}}}& \mathcal{S}_{\ii^{*\text{op}}}(\lambda)
\end{tikzcd} 
\]
Furthermore, the following are commutative diagrams of volume preserving piecewise linear bijections which are linear for $\ii=\ii_0$.
\[ \begin{tikzcd}[row sep=normal, column sep=huge]
\mathcal{L}_{\ii}^{\mathbb{R}}  \arrow{r}{\str_{\ii}\circ b_{\ii}} \arrow{d}{*} & \mathcal{S}_{\ii}^{\mathbb{R}} \arrow{d}{*} \\%
\mathcal{L}^{\mathbb{R}}_{\ii^{*\text{op}}} \arrow{r}{\str_{\ii^{*\text{op}}}\circ b_{\ii^{*\text{op}}}}& \mathcal{S}_{\ii^{*\text{op}}}^{\mathbb{R}}
\end{tikzcd} 
\qquad\qquad 
\begin{tikzcd}[row sep=normal, column sep=huge]
\mathcal{L}_{\ii}(\lambda)^{\mathbb{R}}  \arrow{r}{\str_{\ii}\circ b_{\ii}} \arrow{d}{*} & \mathcal{S}_{\ii}^*(\lambda)^{\mathbb{R}} \arrow{d}{*} \\%
\mathcal{L}^*_{\ii^{*\text{op}}}(\lambda)^{\mathbb{R}} \arrow{r}{\str_{\ii^{*\text{op}}}\circ b_{\ii^{*\text{op}}}}& \mathcal{S}_{\ii^{*\text{op}}}(\lambda)^{\mathbb{R}}
\end{tikzcd} 
\]

\def\cprime{$'$} \def\cprime{$'$} \def\cprime{$'$} \def\cprime{$'$}


\begin{thebibliography}{10}


\bibitem[BFZ96]{BFZ}
Arkady Berenstein, Sergey Fomin, and Andrei Zelevinsky.
\newblock Parametrizations of canonical bases and totally positive matrices.
\newblock {\em Adv. in Math.}, 122 (1996), 49--149.


\bibitem[BZ93]{BZ}
Arkady Berenstein and Andrei Zelevinsky.
\newblock String bases for quantum groups of type {$A_r$}.
\newblock In {\em I. {M}. {G}elfand {S}eminar}, volume~16 of {\em Adv.
	Soviet Math.}, pages 51--89. Amer. Math. Soc., Providence, RI, 1993.

\bibitem[BZ01]{BZ2}
Arkady Berenstein and Andrei Zelevinsky.
\newblock Tensor product multiplicities, canonical bases and totally positive
varieties.
\newblock {\em Invent. Math.}, 143(1):77--128, 2001.

\bibitem[BF16]{BF}
Lara Bossinger and Ghislain Fourier.
\newblock String cone and Superpotential combinatorics for flag varieties and Schubert varieties.
\newblock preprint 2016. arxiv:1611.06504

\bibitem[CMMG04]{CMMG}
Philippe Caldero, Robert Marsh R, Sophie Morier-Genoud. 
\newblock {Realisation of Lusztig cones.} 
\newblock {\em Representation Theory}, 8: 458--478, 2004.

\bibitem[CFL]{CFL}
Rocco Chiriv\`i, Xin Fang, Peter Littelmann. 
\newblock Semitoric degenerations via Newton-Okounkov bodies, LS-algebras and standard monomial theory. 
\newblock{In preparation.}

\bibitem[DKKA07]{DKKA2} V.I.Danilov, A.V.Karzanov, G.A.Koshevoy.
\newblock Combinatorics of regular $A_2$-crystals.
\newblock {\em Journal of Algebra}, 310 (2007), 218--234.

\bibitem[FFL17]{FFL1} Xin Fang, Ghislain Fourier, and Peter Littelmann. 
\newblock Essential bases and toric degenerations arising from
birational sequences. 
\newblock {\em Adv. Math} Volume 312, 107--149, 2017.

\bibitem[D93]{D93} 
M. Dyer.
\newblock Hecke algebras and shellings of Bruhat intervals. \newblock {\em Comp. Math.} 89 (1993), 91--115.

\bibitem[FN17]{FN}	
N. Fujita and S. Naito. 
\newblock Newton-Okounkov convex bodies of Schubert varieties and polyhedral realizations of crystal
bases. 
\newblock {\em Math. Z.} 285 (2017), 325--352.

\bibitem[FO17]{FO}
N. Fujita and H. Oya. 
\newblock A comparison of Newton-Okounkov polytopes of Schubert varieties. 
\newblock {\em J. Lond. Math. Soc.} (2)
96 (2017), 201--227.

\bibitem[G18]{G18} Volker Genz. Crystal combinatorics and mirror symmetry for cluster varieties. Dissertation
2018. url: https://kups.ub.uni-koeln.de/8376/.

\bibitem[GKS16]{GKS1}
Volker Genz, Gleb Koshevoy and Bea Schumann.
\newblock Combinatorics of canonical bases revisited: Type A.
\newblock preprint 2016. arXiv:1611.03465 .

\bibitem[GKS17]{GKS2}
Volker Genz, Gleb Koshevoy, and Bea Schumann. 
\newblock Polyhedral parametrizations of canonical bases \& cluster duality.
\newblock preprint 2017. arXiv:1711.07176

\bibitem[GKS19]{GKS3}
Volker Genz, Gleb Koshevoy, and Bea Schumann. 
\newblock On the interplay of the parametrizations of canonical bases by Lusztig and string data.
\newblock preprint 2018.

\bibitem[GP00]{GP}
Oleg Gleizer and Alexander Postnikov.
\newblock Littlewood-{R}ichardson coefficients via {Y}ang-{B}axter equation.
\newblock {\em Internat. Math. Res. Notices}, (14):741--774, 2000.

\bibitem[H05]{H}
Ayumu Hoshino, Polyhedral realizations of crystal bases for quantum algebras of finite types, J. Math. Phys. 46 (2005),
113514.

\bibitem[J95]{J95} 
Anthony Joseph. 
\newblock Quantum groups and their primitive ideals. 
\newblock {\em Springer-Verlag.} (1995).

\bibitem[K91]{Ka91}
Masaki Kashiwara. 
\newblock On crystal bases of the q--analogue of universal enveloping algebras. 
\newblock {\em Duke Math. J.} 63 (1991), no. 2, 465--516.

\bibitem[K93]{Ka2}
Masaki Kashiwara.
\newblock The crystal base and {L}ittelmann's refined {D}emazure character
formula.
\newblock {\em Duke Math. J.}, 71(3):839--858, 1993.

\bibitem[K94]{Ka3}
Masaki Kashiwara.
\newblock On crystal bases.
\newblock Representations of Groups, {\em Proceedings of the 1994 Annual Seminar of the Canadian Math. Soc.m}, B.N. Allison and G.H. Cliff, eds, {\em CMS Conference Proceedings,} 16:155--197, Amer. Math. Soc.


\bibitem[K15]{K15} 
Kiumars Kaveh. 
\newblock Crystal bases and Newton-Okounkov bodies. 
\newblock {\em Duke Math. J}. 164,
2461--2506, 2015.


\bibitem[Lit94]{Lit2} Peter Littelmann.
\newblock A Littlewood-Richardson rule for symmetrizable Kac-Moody algebras. 
\newblock {\em Inventiones mathematicae} 116.1--3: 329--346,1994.

\bibitem[Lit98]{Lit}
Peter Littelmann.
\newblock Cones, crystals, and patterns.
\newblock {\em Transform. Groups}, 3(2):145--179, 1998.

\bibitem[L90]{Lu}
George Lusztig.
\newblock Canonical bases arising from quantized enveloping algebras.
\newblock {\em J. Amer. Math. Soc.}, 3(2):447--498, 1990.

\bibitem[Lu90]{Lu2}
George Lusztig.
\newblock Finite-dimensional {H}opf algebras arising from quantized universal
enveloping algebra.
\newblock {\em J. Amer. Math. Soc.}, 3(1):257--296, 1990.

\bibitem[MG03]{MG} Sophie Morier-Genoud. 
\newblock Rel\` evement g\' eom\' etrique de la base canonique et involution de
Sch\" utzenberger, 
\newblock {\em C.R. Acad.Sci. Paris. Ser. I337}, 371--374, 2003.

\bibitem[L93]{L93}
George Lusztig.
\newblock Introduction to quantum groups.
\newblock {\em Birkh\"auser}, 1993.

\bibitem[N99]{N99}
Toshiki Nakashima.
\newblock Polyhedral realizations of crystal bases for integrable highest weight modules.
\newblock {\em J. Algebra} 219, 571--597, 1999.

\bibitem[NZ97]{NZ} 
T. Nakashima and A. Zelevinsky. 
\newblock Polyhedral realizations of crystal bases for quantized Kac-Moody algebras. 
\newblock {\em Adv. Math.} 131 (1997), 253--278.

\bibitem[Z13]{Ze}
Shmuel Zelikson.
\newblock On crystal operators in {L}usztig's parametrizations and string cone
defining inequalities.
\newblock {\em Glasg. Math. J.}, 55(1):177--200, 2013.

\end{thebibliography}
\end{document}